\documentclass[a4paper, numbers=endperiod]{scrartcl}

\usepackage{preamble}

\author{Jan McGarry-Furriol}

\title{Rational characteristic classes of bundles with fibre a product of spheres}
\date{\today}

\begin{document}
\maketitle
\vspace*{-14pt}
\begin{abstract}
  We prove the existence of many non-trivial characteristic classes of smooth oriented bundles with fibre a product \( \SnSn \) of odd-dimensional spheres.
  We do so by proving injectivity of the map from the ring of rational characteristic classes of oriented fibrations with fibre \( \SnSn \); the latter is proven by Berglund--Zeman to be isomorphic to the group cohomology of the symmetric powers of the standard representation of a certain finite-index subgroup \( \Gamma \) of \( \mathrm{SL}_{2}(\mathbb{Z}) \).
  These characteristic classes of smooth bundles are not generalised Miller--Morita--Mumford classes, and they exist in arbitrarily large cohomological degrees.
  Inspired by an example given by Morita, we provide a collection of smooth oriented \( \SnSn \)-bundles, indexed by cyclic subgroups of \( \Gamma \), which detect any given non-zero characteristic class of such fibrations.
\end{abstract}
\tableofcontents

\section{Introduction}\label{sec:introduction}

We present results regarding rational characteristic classes of numerable fibre bundles with fibre a product \( \SnSn \) of odd-dimensional spheres, and structure group \( \Diff^+(\SnSn) \), the topological group of orientation-preserving self diffeomorphisms of \( \SnSn \), equipped with the smooth Whitney topology.
We refer to these as smooth oriented \( \SnSn \)-bundles.
Henceforth, \( n \) is an odd natural number.

A rational characteristic class of smooth oriented \( \SnSn \)-bundles is a natural assignment of a rational cohomology class of the base space to each such bundle.
The set of isomorphism classes of smooth oriented \( \SnSn \)-bundles over a space \( B \) is in natural bijection with the set of homotopy classes of maps from \( B \) to the classifying space \( \B\Diff^+(\SnSn) \).
This implies that the set of rational characteristic classes is in natural bijection with
\begin{equation}
  \label{eq:bundle characteristic classes}
  \H^{*}(\B\Diff^{+}(\SnSn);\QQ) .
\end{equation}

Since a fibre bundle is a fibration, a source of such classes are rational characteristic classes of oriented \( \SnSn \)-fibrations.
Similarly to the previous paragraph, the latter is in natural bijection with the rational cohomology
\[
  \H^{*}(\B\aut^+(\SnSn);\QQ)
\]
of the classifying space of the topological monoid of orientation-preserving self homotopy equivalences of \( \SnSn \).
Our main result is the following.

\begin{theorem}\label{intro:injection of characteristic classes}
  The map
  \begin{equation}\label{eq:fibrations to bundles}
    \H^{*}(\B\aut^+(\SnSn);\QQ)
    \lra
    \H^{*}(\B\Diff^+(\SnSn);\QQ)
  \end{equation}
  is injective.
\end{theorem}

We remark the following about the partial description of \eqref{eq:bundle characteristic classes} that we obtain by combining \cite{BerglundZeman2024algebraicmodels} and \cref{intro:injection of characteristic classes}.
\begin{enumerate}[label=(\roman*)]
  \item The domain has been described in~\cite{BerglundZeman2024algebraicmodels}.
        The reduced cohomology is concentrated in odd degrees; in particular, its product is trivial.
        Therefore, it is determined by its Poincaré series, which is given in loc.\ cit., and recalled in \cref{sec:euler-classes-SnSn-fibrations}.
        Over the complex numbers, it can be described in terms of modular forms via the Eichler--Shimura isomorphism.
        It contains non-zero classes in arbitrarily large cohomological degrees.
\end{enumerate}
In relation to terms that are usually involved in describing rings of characteristic classes of manifold bundles, we observe the following.
We come back to introduce these terms by the end of the introduction.
\begin{enumerate}[resume, label=(\roman*)]
  \item The image of the map \eqref{eq:fibrations to bundles} does not intersect the subring generated by the generalised Miller--Morita--Mumford classes in positive cohomological degrees. Indeed, the latter is concentrated in even degrees.
  \item The obtained partial description of \eqref{eq:bundle characteristic classes} applies in all cohomological degrees; in particular, beyond the pseudoisotopy stable range.
\end{enumerate}

Our second result provides more information on the evaluation of these characteristic classes.
While it implies \cref{intro:injection of characteristic classes}, we derive both from a third statement about Torelli spaces given in \cref{sec:injection of characteristic classes}.
For simplicity, we do not state the latter in the introduction.

Let \( \Gamma \) be \( \SL_{2}(\ZZ) \) if \( n \) is 1, 3 or 7, and the theta subgroup otherwise, that is the subgroup consisting of matrices with exactly one odd entry in each row.
Let \( \langle{\gamma}\rangle \le \Gamma \) denote the subgroup generated by an element \( \gamma \) of \( \Gamma \).
Let
\[
  \Theta_{\gamma}\cl
  E_{\gamma}
  \lra
  B_{\gamma}
\]
be the universal smooth oriented \( \SnSn \)-bundle over the cover \( B_{\gamma} \) of \( \B\Diff^+(\SnSn) \) corresponding to the subgroup of mapping classes that induce an element in \( \langle{\gamma}\rangle \) in degree \( n \) homology.
We prove that this family of bundles detects any given characteristic class of oriented \( \SnSn \)-fibrations by using their relation to the decomposable homology classes which concern the results in~\cite{GoldmanMillson1986decomposable}.

\begin{theorem}\label{intro:pairing decomposable fibrations}
  The set of smooth oriented \( \SnSn \)-bundles \( \{\Theta_{\gamma}\mid \gamma\in \Gamma\} \) satisfies the following properties.
\begin{enumerate}
    \item\label{intro:item:decomposable detection} For each non-zero rational characteristic class \( c \) in \( \H^*(\B\aut^+(\SnSn);\QQ) \) there exists an element \( \gamma \) in \( \Gamma \) such that \( c(\Theta_{\gamma}) \) is non-zero.
   \item\label{intro:item:ev for parabolic}
         If \( \gamma \) is parabolic (see \cref{def:parabolic subgroup}), then there exists a rational characteristic class \( c \) such that \( c(\Theta_{\gamma}) \) is non-zero.
\end{enumerate}
\end{theorem}

In \cref{sec:E2k}, we construct explicit characteristic classes \( \E_{2k} \) that satisfy Part~\ref{intro:item:ev for parabolic} of \cref{intro:pairing decomposable fibrations}.
This is closely related to an example by Morita~\cite[Example 5.3]{Morita1987t2bundles}, as we now briefly explain.

Morita studied the ring of characteristic classes of smooth \( \S^1\times\S^1 \)-bundles. (In this case, \eqref{eq:fibrations to bundles} is an isomorphism.)
In particular, Morita shows in loc.\ cit.\ how one can construct a smooth oriented \(\S^1\times\S^{1}\)-bundle over \( \S^1\times \CP^{\ell} \) that witnesses the non-triviality of a certain family of these classes.
We extend the construction of Morita as follows.

Let \( Q \) be the subgroup of \( \Gamma \) generated by \(
\begin{psmallmatrix}
  2&-1\\1&0
\end{psmallmatrix}
\).
Let \( P \) be the subgroup \( \mathrm{U}\cap\Gamma \), where \( \mathrm{U} \) is the subgroup of \( \SL_2(\ZZ) \) generated by \(
\begin{psmallmatrix}
  1&1\\0&1
\end{psmallmatrix}
\).
Both \( P \) and \( Q \) are parabolic subgroups; in particular they are infinite and cyclic.
For each natural number \( k \), we construct a class
\[
  \mathscr{E}_{2k} \in \H^{2kn+2k+1}(\B\aut^+(\SnSn);\QQ).
\]
If \( n \) is 1, these are the classes that are featured in Morita's example.
We provide smooth oriented \( \SnSn \)-bundles
\begin{gather*}
  \label{eq:1}
  M_Q(1,1)
  \lra
  \B Q\times\B\SO(n+1),
  \\
  M_P\lra \B\Dehn
  \label{eq:sphere bundle with bundle data}
\end{gather*}
where \( \Dehn \) denotes the topological group of smooth maps equipped with the smooth Whitney  topology and the pointwise product.
These bundles admit a map to \( \Theta_Q \) and \( \Theta_P \), respectively.
In particular, the following result gives an example of \cref{intro:pairing decomposable fibrations}, Part~\ref{intro:item:ev for parabolic}.

\begin{proposition}
  The classes \(\mathscr{E}_{2k}(M_Q(1,1)) \) and \( \E_{2k}(M_{P})\) are non-zero for each \( k \).
\end{proposition}

When \( n \) is 3, we can restrict \( M_{P} \) to \( \B P\times\mathbb{H}\mathrm{P}^{\ell} \) while preserving the property that \( \mathscr{E}_{2k} \) pulls-back to a non-zero class for all \( 2k\le\ell \).

\paragraph{Relation to other results}
Consider the \( g \)-fold connected sum \( W_g=\#_{g}\SnSn\), and the manifold \( W_{g,1} \) obtained by removing an open \( 2n \)-dimensional disc from \( W_g \).
This family of manifolds plays a distinguished role in the study of characteristic classes of manifold bundles, and it has been studied extensively.
We explain which of the results in the literature about this family of examples do and do not apply to the ring
\[
  \H^{*}(\B\Diff^{+}(\SnSn);\QQ).
\]

Let \( M \) be a simply connected compact smooth manifold of dimension \( m \).
Let \( \Diff_{\partial}(M) \) be the topological group consisting of those diffeomorphisms that fix a neighborhood of the boundary pointwise.
There has been important progress in the last couple of decades on the understanding of the cohomology of \( \B\Diff_{\partial}(M) \), notably sparked by \cite{MadsenWeiss2007} (in the case \( m=2 \)), and \cite{GRW2014moduli}, \cite{GRW2018stabilityI}, and \cite{GRW2017stabilityII} (in the case \( m\ge6 \) even).
In the cited work, the authors provided a homotopy theoretical description of the cohomology of \( \B\Diff_{\partial}(M) \), in a range of degrees depending on the \textit{genus} of \( M \); the latter is the largest natural number for which \( M \) admits an embedding of \( W_{g,1} \).
In the case of \( W_g \) and \( W_{g,1} \), this range of degrees goes up to \( (g-3)/2 \), and the answer is given by the free graded commutative algebra on certain classes
\begin{equation}\label{eq:MMM classes}
  \kappa_{b}\in\H^{|b|-m}(\B\Diff_{\partial}(M);\QQ).
\end{equation}
These are called generalised Miller--Morita--Mumford classes, and they are assigned to monomials \( b \) in the Euler class and Pontryagin classes in \( \H^{*}(\B\SO(2n);\QQ) \).
However, this description does not apply to the case of genus one, which we treat in this paper.
For the same reason, results such as \cite{KupersRandalWilliams2020algebraicTorelli} and \cite{KRW2020torelliCalc} do not apply.
Nevertheless, the mentioned work of Galatius--Randal-Williams is involved in results of \cite{goodwillieStabilityConcordanceEmbeddings2024} and \cite{Kupers2019finiteness}, which do provide some information, as we explain below.

More classical is the pseudoisotopy approach to the homotopy type of the diffeomorphisms group of \( M \), for \( m\ge5 \).
In this approach, one considers the map
\[
  \B\Diff_{\partial}(M)
  \lra
  \B\widetilde{\Diff}_{\partial}(M)
\]
to the classifying space of the so-called block diffeomorphisms group.
By \cite{weissAutomorphismsManifoldsAlgebraic1988} the homotopy fibre of this map can be described in terms of Algebraic K-theory, in those homological degrees within the pseudoisotopy (or concordance) stable range.
In the case of \( \SnSn \), it follows from \cite{goodwillieStabilityConcordanceEmbeddings2024} that this range of degrees in rational homology goes up to degree \( 2n-4 \), as long as \( n\ge 5 \).
The remaining problem is to access the rational cohomology of \( \B\widetilde{\Diff}_{\partial}(M) \).
To do this, one can analyse \( \B\aut_{\partial}(M) \), and use the surgery exact sequence.
An example of an execution of this programme is the work in~\cite{berglundHomologicalStabilityDiffeomorphism2013} and~\cite{BerglundMadsen2020}, which treats the case of \( W_{g,1} \).
See also~\cite{KrannichKupersoberwolfach2019homotopyTheory}.
Their results do not provide information about the ring of characteristic classes that we study in this paper, because the map \( \B\aut_{\partial}(W_{1,1})\ra \B\aut^+(W_{1}) \) vanishes on rational cohomology if \( 2n\ge 6 \) (see \cref{rmk:vanishing of partial Wg1 to Wg}).
However, they serve us to contrast with \cref{intro:injection of characteristic classes}.
If we combine \cite{BerglundMadsen2020} with the Euler characteristic estimates of \cite{BorinskyVogtmann2023EulerCharacteristic}, we learn that there is a huge amount of odd-degree cohomology classes in the domain of the map
\begin{equation*}
  \label{eq:Wg1 fibrations}
  \H^{*}(\B\aut_{\partial}(W_{g.1});\QQ)
  \lra
  \H^{*}(\B\Diff_{\partial}(W_{g,1});\QQ),
\end{equation*}
for \( 2n\ge6\) and large enough \( g \).
These classes must lie in the kernel of this map, because the classes \( \kappa_{b} \) in \eqref{eq:MMM classes} live in even degrees.
It is also mentioned in \cite{KrannichKupersoberwolfach2019homotopyTheory}, but not discussed, that there is an approach based on embedding calculus to prove existence of non-zero classes in the target of \eqref{eq:Wg1 fibrations} beyond this range.

What was known about the ring
\[
  \H^{*}(\B\Diff(\SnSn);\QQ)
\]
is that in each degree it consists of a finite-dimensional vector space \cite{Kupers2019finiteness}, and that the tautological ring modulo the nilradical consists only of the rational numbers in degree zero \cite{galatius2017tautological}.
The tautological ring is the subring generated by the generalised Miller--Morita--Mumford classes.
The part of the tautological ring that is defined in \( \H^{*}(\B\aut^+(\SnSn);\QQ) \) is also \( \QQ \) by~\cite{prigge2024tautologicalringsfibrations}.

\paragraph{Acknowledgments}
I am grateful to my PhD advisor, Alexander Berglund, for suggesting this topic, and for his guidance and help to improve the present work.
I also thank Miguel Barata, Jo\~{a}o Lobo Fernandes, Dan Petersen, Nils Prigge, and Robin Stoll for useful conversations, and for comments on earlier versions of this document.

\section{Fibrations, bundles, and classifying spaces}\label{sec:class-fibr-1}
We present the classification of oriented fibrations using the formalism of~\cite{May1975fibrations}.
By a space we mean a compactly generated weakly Hausdorff space with the homotopy type of a CW complex.

Let \( F \) be a finite CW complex together with a Poincaré duality space structure of dimension \( m \).
That is, an integral homology class \( [F]\) of degree \( m \), such that the cap product with it
\[
  \H^{*}(F;\ZZ)\lra\H_{m-*}(F;\ZZ)
\]
is an isomorphism.
Such a class \( [F] \) is called an orientation on \( F \).

A fibration with fibre \( F \), or \( F \)-fibration, is a (Hurewicz) fibration \( E\ra B \) such that for every point \( b\in B \) there exists a homotopy equivalence \( F\ra E_b \) to the fibre over \( b \).

\begin{definition}
  An oriented \( F \)-fibration is an \( F \)-fibration \( E\ra B \) together with a Poincaré duality space structure on \( E_b \) for every \( b\in B \), such that, for every path in \( B \), the induced map between the respective fibres preserves the orientation.
\end{definition}

Given an \( F \)-fibration over \( B \), we write \( \mathcal{H}_m \) for the local system that takes a point of \( B \) to the \( m \)th homology of the fibre over \( b \).
We call it the monodromy local system.
An orientation on the fibration, with fibre of dimension \( m \), provides an isomorphism between \( \mathcal{H}_m \) and the constant functor \( \H_m(F;\ZZ) \).

A fibre homotopy equivalence between two oriented \( F \)-fibrations \( E\ra B \) and \( E'\ra B \) is a
a map \( f\cl E\ra E' \) of over \( B \) that is fibrewise orientation-preserving, and that admits an inverse up to homotopy over \( B \).

Let \( \B \) be the geometric bar construction of~\cite[\sec 7]{May1975fibrations}.
It is a functor on the category of triples \( (M,G,N) \) where \( G \) is a topological monoid, \( M \) is a right \( G \)-module, and \( N \) is a left \( G \)-module.
A map to another such triple \( (M',G',N') \) in this category is a triple \( (r,\varphi,l) \) where \( \varphi \) is a map between the topological monoids, and \( r \) and \( l \) are module maps to the \( G \)-modules \( \varphi^{*}M' \) and \( \varphi^{*}N' \), respectively.
We omit an entry on a triple when it is the terminal object.
All actions in this paper are left actions, unless specified otherwise.

\begin{theorem}
  Let \( B \) be a space.

  The set of equivalence classes of oriented \( F \)-fibrations over \( B \) is in natural bijection with the set of homotopy classes \( [B,\B\aut^{+}(F)] \).

  Let \( G \) be a topological group that acts effectively on \( F \).
  The set of equivalence classes of numerable fibre bundles over \( B \) with fibre \( F \) and structure group \( G \) is in natural bijection with the set of homotopy classes \( [B,\B G] \).
\end{theorem}
\begin{proof}
  The statement about fibrations follows from~\cite[Theorem 9.2]{May1975fibrations} applied to the category of fibres consisting of Poincaré duality spaces with the homotopy type of \( F \), and orientation-preserving maps; see also~\cite[Corollary 9.5]{May1975fibrations}.
  The statement about fibre bundles is classical, and is also recovered in~\cite[Theorem 9.10]{May1975fibrations}.
\end{proof}

Let \( M \) be an oriented closed smooth manifold.
\begin{definition}
A smooth oriented \( M \)-bundle is a numerable fibre bundle with fibre \( M \) and structure group \( \Diff^+(M) \), the topological group of orientation-preserving diffeomorphisms of \( M \), equipped with the smooth Whitney topology.
\end{definition}

To a map of spaces \( p\cl E\ra \B\pi \), where \( \pi \) is a topological group, we may associate the pullback
\[
  \begin{tikzcd}
    \operatorname{F}(E)
    \ar[r]
    \ar[d]
    &E
    \ar[d,"p"]
    \\
    \mathrm{E}\pi
    \ar[r]
    &\B\pi
  \end{tikzcd}
\]
with the action of \( \pi \) arising from that on \( \mathrm{E}\pi=\B(\pi,\pi,*) \) by left multiplication on \( \pi \), which is simultaneously considered as a right \( \pi \)-space by right multiplication.
The map \( \mathrm{E}\pi\ra\B\pi \) is a numerable principal \( G \)-bundle~\cite[Theorem 8.2]{May1975fibrations}.
In particular, \( \operatorname{F}(E) \) is the homotopy fibre of \( p \).

\begin{definition}
  We call the \( \pi \)-space \( \operatorname{F}(E) \) the monodromy action of \( p \).
\end{definition}

The map \( \mathrm{E}\pi\ra \B\pi \) may be identified with the path-space fibration over \( \B\pi \)~\cite[Proposition 8.7]{May1975fibrations}.
The path-space has an action up to homotopy of \( \pi \), which agrees with the \( \pi \)-action on \( \mathrm{E}\pi \).
If \( p \) is a fibration, then the map \( F\ra \operatorname{F}(E) \) from its fibre is a weak homotopy equivalence, and it is compatible with the actions up to homotopy of \( \pi \).

Conversely, to a space \( X \) with an action of \( \pi \) we may associate the map on homotopy orbits
\[
  X_{\h\pi}=\B(*,\pi,X)
  \lra
  \B\pi.
\]
We call a map of \( \pi \)-spaces a weak equivalence if it is a weak homotopy equivalence between the underlying spaces.
These functors provide an equivalence of homotopy theories, but in this paper we only use the proof of following statement.

\begin{theorem}\label{spaces/BG = G-spaces}
  The homotopy fibre and homotopy orbits functors induce an equivalence between the homotopy category of spaces over \( \B\pi \) and the homotopy category of \( \pi \)-spaces.
\end{theorem}
\begin{proof}
  The functors \( \operatorname{F} \) and \( \B \) preserve weak equivalences, the latter by Proposition 7.3.
  There are natural weak equivalences \( \operatorname{F}(X_{\h\pi})\ra X \), and \( \operatorname{F}(E)_{\h \pi}\ra E \) by Proposition 7.5 and Remarks 8.9, respectively.
\end{proof}

\begin{lemma}\label{conjugation=monodromy}
  Let \( 1\ra G' \ra G \ra[p] G'' \ra 1 \) be a split short exact sequence of group-like topological monoids.
  Suppose that \( G'' \) is a topological group.
  Then the conjugation action of \( G'' \) on \( \B G' \) models the monodromy action of \( \B p \).
\end{lemma}
\begin{proof}
  We have that the monodromy of \( \B p \) is the \( G'' \)-space
  \[
    \B(G'',G'',*)\times_{\B G''}\B G \cong\B(G'',G,*)
  \]
  because geometric realisation preserves finite limits.
  The map \( \B G'\ra \B(G'',G,*) \) is a \( G'' \)-equivariant weak homotopy equivalence for the conjugation action on the domain, and the action on the codomain defined as follows.
  An element \( g \) of \( G'' \) acts by the map \( \B(\ell_{g}\circ \mathrm{r}_{g},\mathrm{c}_{g}) \) induced by the following group homomorphism, and maps of right \( G \)-spaces
   \[
     G\ra[\mathrm{c}_{g}] G,\quad h\mapsto ghg^{-1},\qquad G''\ra[\mathrm{r}_{g}] \operatorname{c}_{g}^{*}(G''),\quad y\mapsto yg^{-1},\qquad G''\ra[\ell_{g}] \operatorname{c}_{g}^{*}(G''),\quad y\mapsto gy.
   \]

   By~\cite[Lemma 3.18]{BerglundZeman2024algebraicmodels}, there is a zig-zag of \( G'' \)-equivariant homotopy equivalences between \( \B(G'',G,*) \) with the action \( \B(\mathrm{r}_{g},\mathrm{c_{g}}) \), and the same space with the trivial action.
   One can check that it can be made a \( G'' \)-equivariant zig-zag with the actions \( \B(\ell_g,\id) \).
\end{proof}

We end this section by introducing notation for covers of \( \B\aut(F) \) and \( \B\Diff(M) \).

Let \( \Gamma(F) \) be the group of automorphisms of the integral homology of \( F \) realised by self homotopy equivalences.
Given a subgroup \( K\le\Gamma(F) \), we denote by \( \aut^K(F) \) the topological submonoid of \( \aut(F) \) consisting of those path components of maps that induce an element of \( K \) in homology.
Its classifying space is the homotopy pullback
\[
  \begin{tikzcd}
    \B\aut^K(F)
    \ar[r]
    \ar[d]
    &\B\aut(F)
    \ar[d]
    \\
    \B K
    \ar[r]
    &\B\Gamma(F).
  \end{tikzcd}
\]
When \( K=\{\id\} \), we instead write \( \B\htor(F) \), and we call it the homotopy Torelli space.

Similarly, we use the notation
\[
  \B\Diff^K(M)\quad\text{and}\quad\B\Tor(M)
\]
for the analogously defined classifying spaces, and we call \( \B\Tor(M) \) the Torelli space.

\section{Rational dg Lie algebra models for spaces}\label{sec:geom-real-diff}
We present the geometric realisation functors that we use from appropriate  subcategories of differential graded Lie algebras (\( \mathsf{dgl} \)), and commutative differential graded algebras (\( \mathsf{cdga} \)), all over the rational numbers.
These functors take values in the category of simplicial sets.
We may compose them with the geometric realisation functor \( |-| \) to the category of compactly generated weakly Hausdorff spaces.

An action of a dg Lie algebra \( \ggg \) on a dg Lie algebra \( L \) is a dg Lie algebra homomorphism
\[
  \rho
  \cl
  \ggg
  \lra
  \Der(L)
\]
to the dg Lie algebra of derivations of \( L \).
The semidirect product of \( \ggg \) and \( L \) with respect to \( \rho \) is the dg Lie algebra \(L\rtimes\ggg\) with underlying chain complex \( L\oplus\ggg \), and the bracket given by
\begin{gather*}
  [(\theta,x),(\delta,y)]=([\theta,\delta],[x,y]+\rho(\theta)(x)-{(-1)}^{|y||\delta|}\rho(\delta)(y)).
\end{gather*}
The action is nilpotent if for every degree \( m \) there is a natural number \( k \) such that the \( k \)th composition \( \rho{(\ggg_0)}^{\circ k}(L_{m}) \) is zero; here, \( \ggg_0 \) is the degree zero part of \( \ggg \).
We say that \( \ggg \) is nilpotent if its adjoint action on itself is nilpotent.
That is, if its lower central series terminates at a finite stage in each degree.
Given a dg Lie algebra \( \ggg \) with differential \( d \), we denote by \( \ggg\langle{0}\rangle \) the dg Lie subalgebra
\[
  \ggg\langle{0}\rangle_{k}=
  \left\{
  \begin{array}{lll}
    \ggg & \text{if \( k>0 \)},\\
    \mathrm{Z}_{0}(\ggg) & \text{if \( k=0 \)},\\
    0& \text{if \( k<0 \)},
  \end{array}
  \right.
\]
where \( \mathrm{Z}_0(\ggg) \) are the \( 0 \)-cycles.
We say that \( \ggg \) is connected if \( \ggg\langle{0}\rangle =\ggg \).

If \( \ggg \) is a nilpotent Lie algebra, we denote by \( \exp\ggg \) the nilpotent group with the same underlying set as \( \ggg \), and with the group operation given by the Baker--Campbell--Hausdorff formula.
Following \cite[\sec 3.2]{Berglund2020autfiberbundles}, given a connected nilpotent dg Lie algebra \( \ggg \), we define a simplicial nilpotent group by
\[
  \exp_{\bullet}\ggg=\exp \mathrm{Z}_0(\ggg\otimes\Omega_{\bullet})
\]
where \( \Omega_{\bullet} \) is the simplicial cdga that in degree \( n \) consists of polynomial differential forms on the standard \( n \)-simplex.
We obtain a nilpotent topological group \(|\exp_{\bullet}\ggg|\) by applying geometric realisation, because \( |-| \) preserves finite products.

For the remainder of this section, we let \( \ggg \) and \( L \) be connected nilpotent dg Lie algebras.
A nilpotent action of \( \ggg \) on \( L \) induces a simplicial group action of \( \exp_{\bullet}\ggg \) on the simplicial dg Lie algebra \( L\otimes\Omega_{\bullet} \) by taking the exponential series of a derivation \( \theta \),
\[
  e^{\theta}(x)= x+ \sum_{n\ge1}\frac{1}{n}\rho{(\theta)}^{\circ n}(x).
\]
Thus, it acts on \( \B\exp_{\bullet}(L) \).

Similarly, a nilpotent action of \( \ggg \) on a cdga \( A \) by derivations induces an action of \( \exp_{\bullet}\ggg \) on the realisation of \( A \)
\[
  \langle{A}\rangle
  =\Hom_{\mathsf{cdga}}(A,\Omega_{\bullet}) \cong\Hom_{\mathsf{cdga}(\Omega_{\bullet})}(\Omega_{\bullet}\otimes A,\Omega_{\bullet})
\]
by precomposition by the automorphisms of \( \Omega_{\bullet}\otimes A \), as a simplicial cdga over \( \Omega_{\bullet} \), obtained by taking the exponential series of a derivation.
We denote by \( \CE(\ggg) \) the Chevalley--Eilenberg cochains algebra, and by \( \CE(\ggg;A) \) that with coefficients in \( A \).
See~\cite{Berglund2022characteristicClassesFamiliesBundles} for the definition.
\begin{definition}
  A nilpotent action of \( \mathfrak{g} \) on \( L \) (respectively, \( A \)) models the action of a topological monoid \( G \) on a space \( X \) if \( (G,X) \) is rational homology equivalent to
  \[
    (|\exp_{\bullet}(\ggg)|,|\B\exp_{\bullet}(L)|)
    \qquad\text{(respectively, \( (|\exp_{\bullet}(\ggg)|,\langle A\rangle)  \))}.
  \]
  in the category of triples defined in \cref{sec:class-fibr-1}.
\end{definition}

\begin{proposition}
  \label{exp and conjugation}\label{B and homotopy orbits}
  Suppose that \( A \) is degreewise finite dimensional.
  A nilpotent action \( \rho \) of \( \ggg \) on \( L \) (respectively \( A \)) is a model for the monodromy of the induced map
  \[
    \B \exp_{\bullet}(L\rtimes\ggg)
    \lra
    \B\exp_{\bullet}(\ggg)
    \qquad\text{(respectively, \( \langle \CE(\ggg;A)\rangle \ra \langle{\CE(\ggg)}\rangle  \))}.
  \]
\end{proposition}
\begin{proof}
  See~\cite[Proposition 3.6]{Berglund2022characteristicClassesFamiliesBundles} for the statement about cdga's.

  A proof using the functor \( \exp_{\bullet} \) that we use in this paper is the following.
  We have a split short exact sequence
  \[
    \begin{tikzcd}
      1
      \ar[r]
      &\exp_{\bullet}(L)
      \ar[r]
      &\exp_{\bullet}(L\rtimes\ggg)
      \ar[r,shift left,"p"]
      &\exp_{\bullet}(\ggg)
      \ar[l,shift left]
      \ar[r]
      &1
    \end{tikzcd}
  \]
  because \( \mathrm{Z}_0(-\otimes\Omega_{\bullet}) \) is left exact and functorial.
  Let \( x \) and \( y \) be \( q \)-simplices of \( \exp_{\bullet}(\ggg) \) and \( \exp_{\bullet}(L) \), respectively.
  The Campbell identity
  \[
    x*y*x^{-1}=e^{\mathrm{ad}_{x}}(y)
  \]
  shows that conjugating \( y \) by \( x \) in \( \exp_{\bullet}(L\rtimes\ggg) \) agrees with the action of \( x \) on \( y \), because the bracket in the semidirect product is defined so that \( \mathrm{ad}_{x}=\rho(x) \).
  The result now follows by \cref{conjugation=monodromy}.
\end{proof}

\begin{remark}
  By the proof of \cref{conjugation=monodromy}, and the universal property of the semidirect product we obtain a natural isomorphism of simplicial groups
  \[
    \exp_{\bullet}(L)\rtimes\exp_{\bullet}(\ggg)
    \lra
    \exp_{\bullet}(L\rtimes\ggg)
  \]
  that it is given degreewise by sending \( (y,x) \) to \( y*x \).
\end{remark}
\section{Group cohomology}\label{sec:group-cohom-homol}

In this section, we provide the necessary preliminaries on group cohomology.
We refer to \cite{Brown1994groupCohomology} for more details.
Note that the group cohomology of the discrete group \( \Gamma \) agrees with the singular cohomology of \( \B\Gamma \).

Let \( R \) be a commutative ring.
Our convention is that modules over the group ring \( R[\Gamma] \) will be left modules. We denote by \( \Cat{Mod}_{R[\Gamma]} \) their category.
Given an \( R[\Gamma] \)-module \( V \), consider the \( R \)-modules
\begin{equation*}
  V^{\Gamma}=\{v\in V \mid \forall \gamma\in\Gamma,\ \gamma\cdot v=v \}\quad\text{and}\quad V_{\Gamma}= V/\langle{\gamma\cdot v-v\mid \gamma\in\Gamma,\ v\in V}\rangle
\end{equation*}
of its invariants and coinvariants, respectively.
Here, the brackets \( \langle{-}\rangle \) denote the submodule generated by a set.
These constructions are functorial.
For each natural number \( k \), the \( k \)th cohomology and homology functors of \( \Gamma \) over \( R \) are defined as the \( k \)th derived functors
\begin{equation*}
  \H^k(\Gamma;-),\ \H_k(\Gamma;-)\cl\Cat{Mod}_{R[\Gamma]}\lra \Cat{Mod}_R
\end{equation*}
of the invariants and coinvariants functors, respectively.

In particular, \( \H^0(\Gamma;V)=V^{\Gamma} \) and \( \H_0(\Gamma;V)=V_{\Gamma} \).
The first cohomology group can be described as follows.
Recall that a derivation for a left \( R[\Gamma] \)-module \( V \) is a function \( d\cl \Gamma\ra V \) such that \( d(g\cdot h)=d(g)+g\cdot d(h) \) for every \( g,h\in \Gamma \).
There is an isomorphism
\begin{equation*}
  \label{eq:derivations}
  \H^1(\Gamma;V)\cong\operatorname{Der}_R(\Gamma,V)/\operatorname{P}(\Gamma,V)
\end{equation*}
with the quotient of the \( R \)-module of derivations by the submodule of derivations of the form \( d(g)=v-g\cdot v \) for some \( v\in V \).

Given a subgroup \( G\le\Gamma \), there are natural transformations
\begin{align*}
  \operatorname{res}^{\Gamma}_G\cl &\H^{k}(\Gamma;-) \lra \H^k(G;-),
  \\
  \operatorname{cores}^{\Gamma}_G\cl &\H_{k}(G;-) \lra \H_{k}(\Gamma;-),
\end{align*}
as explained in~\cite[\nopp III.8]{Brown1994groupCohomology}.
Group cohomology is functorial in pairs of a group homomorphism \( \varphi\cl G'\ra G \), and a map of \( R[G'] \)-modules \( V' \ra \varphi^{*}V \).

\begin{remark}\label{evaluation of invariants (algebraic)}
  A pair of elements \( \gamma\in\Gamma \) and \( v\in V \) such that \( \gamma\cdot v=v \) yields a homology class \( v\cap \gamma \) under the cap product
  \begin{equation}
    \cap\cl\H^k(G;V)\otimes_{R}\H_1(G;R)\lra \H_{1-k}(G;V)\label{eq:cap invariants}
  \end{equation}
  for \( k=0 \), and \( G=\langle{\gamma}\rangle\le\Gamma \) the subgroup generated by \( \gamma \).
  Following~\cite{GoldmanMillson1986decomposable}, we call homology classes of the form
  \[
    \operatorname{cor}^{\Gamma}_{\langle{\gamma}\rangle}(v\cap\gamma)
  \]
  in \( \H_1(\Gamma;V) \) decomposable homology classes.
\end{remark}

\section{Symmetric powers of the standard two-dimensional representation}\label{sec:symm-powers-stand}

In this section, we recall three results about the group cohomology of the symmetric powers of standard two-dimensional representation of \( \Gamma\le\SL_2(\ZZ) \) and its dual.
These are used in \cref{sec:decomp-bund} to prove \cref{intro:pairing decomposable fibrations}.
Precisely, we consider the following two representations.

Let \( \QQ[\epx,\epy] \) be the polynomial algebra in two variables with the action of \(
\begin{psmallmatrix}
  a&b\\c&d
\end{psmallmatrix}
\) on a polynomial \( f \) given by
\[
  \begin{psmallmatrix}
    a&b\\c&d
  \end{psmallmatrix}
  f(\epx,\epy)
  =
  f\left(
  (\epx, \epy)
  \begin{psmallmatrix}
    a&b\\c&d
  \end{psmallmatrix}
  \right)
  =f(a\epx+c\epy, b\epx+d\epy).
\]
The summand \( \QQ[\epx,\epy]_k \) consisting of polynomials of degree \( k \) is isomorphic to the \( k \)th symmetric power \( \operatorname{Sym}^k\QQ^{2} \) of the standard representation \( \QQ^2 \).

Let \( \QQ[\ex,\ey] \) be the polynomial algebra with the action of \(
\begin{psmallmatrix}
  a&b\\c&d
\end{psmallmatrix}
\) on a polynomial \( f \) given by
\[
  \begin{psmallmatrix}
    a&b\\c&d
  \end{psmallmatrix}
  f(\ex,\ey)
  =
  f\left(
    (\ex,\ey)
  \begin{psmallmatrix}
    d&-c\\-b&a
  \end{psmallmatrix}
  \right)
  =f(d\ex-b\ey,-c\ex+a\ey).
\]
The summand \( \QQ[\ex,\ey]_k \) is isomorphic to the \( k \)th symmetric power of the dual of the standard representation.

Let
\(
\gamma=
\begin{psmallmatrix}
  a&b\\c&d
\end{psmallmatrix}
\) be in \(\SL_2(\ZZ) \).
The element
\[
  s_{\gamma}=b\epx^{2}+(d-a)\epx\epy-c\epy^2
\]
is invariant under \( \gamma \).
It gives rise to a decomposable homology class as in \cref{evaluation of invariants (algebraic)}.

\begin{theorem}[{\cite[\nopp Theorem B]{GoldmanMillson1986decomposable}}]\label{decomposables generation}
  For each \( m\ge1 \), the set of decomposable classes
  \[
    \{ \operatorname{cor}^{\Gamma}_{\langle{\gamma}\rangle}(s^m_{\gamma}\cap\gamma) \mid \gamma \in\Gamma\}
  \]
  spans \( \H_1(\Gamma; \QQ[\epx,\epy]_{2m}) \).
\end{theorem}

\begin{definition}\label{def:parabolic subgroup}
  We say that a non-trivial subgroup \( P\le\Gamma \) is parabolic if it is of the form \( \Gamma\cap g\mathrm{U}g^{-1} \) for some \( g\in \SL_2(\RR) \), where
  \( \mathrm{U}\le\SL_2(\ZZ) \)
is the subgroup of upper-triangular matrices with \( 1 \) in all diagonal entries.
We say that an element \( \gamma\in\Gamma \) is parabolic if \( \gamma\ne 1 \) and it is of the form \( gug^{-1} \) for some \( g\in \SL_2(\RR) \) and \( u\in \mathrm{U} \).
\end{definition}

Given a ring \( R \) and an \( R[\Gamma] \)-module \( V \), we define the parabolic subspace \( \H^1_{\mathrm{par}}(\Gamma;V) \) as the subspace of \( \H^1(\Gamma;V) \) consisting of classes which vanish upon restriction to every parabolic subgroup.
The quotient by the parabolic subspace can be described as follows when \( V \) is a symmetric power of the standard representation.

Consider the standard action of \( \Gamma \) on the projective line \( \mathrm{P}^1(\QQ) \).
We denote by \( \Gamma_p \) the stabiliser of a point \( p \) of \( \mathrm{P}^1(\QQ) \), and by \( \Gamma_p^+ \) the subgroup consisting of matrices with positive trace.
The parabolic subgroups of \( \Gamma \) are exactly the subgroups of \( \Gamma_p^+ \) as \( p \) varies in \( \mathrm{P}^1(\QQ) \).
For example, \( \SL_2(\ZZ)_{[1:0]}=\mathrm{U}\times\{\pm 1\} \).
We have the following well-known result (see \cite[\sec 4]{Wiese2019modularForms}).

\begin{theorem}\label{parabolic ses}
  Restriction to the stabilisers of the action of \( \Gamma \) on \( \mathrm{P}^1(\QQ) \) gives a short exact sequence
  \[
    0
    \ra
    \H^{1}_{\mathrm{par}}(\Gamma;\QQ[\ex,\ey]_{2m})
    \ra
    \H^1(\Gamma;\QQ[\ex,\ey]_{2m})
    \ra
    \bigoplus_{\Gamma\cdot p\in\Gamma\backslash \mathrm{P}^1(\QQ)}\H^1(\Gamma_{p};\QQ[\ex,\ey]_{2m})
    \ra
    0
  \]
  of rational vector spaces where each summand of the quotient is one-dimensional.
\end{theorem}

The virtual cohomological dimension of a group is the common cohomological dimension of its torsion-free subgroups of finite index.
The following result is also well-known.

\begin{theorem}\label{vcd of Gamma}
  The virtual cohomological dimension of \( \Gamma \) is \( 1 \).
  In particular, \( \H^{*}(\Gamma;V)=0 \) for all \( *>1 \) and all \( \QQ[\Gamma] \)-modules \( V \).
\end{theorem}
\begin{proof}
  The virtual cohomological dimension of \( \SL_2(\ZZ) \) is 1 because it contains a free subgroup of rank 2 with index 12; namely, the commutator subgroup (see \cite[p. 248]{Brown1994groupCohomology}).
  Since \( \Gamma \) has finite index in \( \SL_2(\ZZ) \), the first statement follows.
  The second statement is implied by using~\cite[Proposition 10.1]{Brown1994groupCohomology}
\end{proof}

\section{Characteristic classes of fibrations with fibre a product of spheres}\label{sec:euler-classes-SnSn-fibrations}

Let \( \xSn \) denote the \( d \)-fold product of \( \S^n \).

In this section, we define \textit{Euler classes} associated to an oriented \( \xSn \)-fibration equipped with a trivialisation of the local system given by the homology of the fibres.
Then we provide the description given by Berglund--Zeman~\cite{BerglundZeman2024algebraicmodels} of the ring of rational characteristic classes of oriented \( \xSn \)-fibrations in terms of these Euler classes, and an action of an arithmetic group.

We start by recording the group of automorphisms of the \( n \)th homology of \( \xSn \) realised by self homotopy equivalences and diffeomorphisms.
We choose a basepoint of \( \xSn \).
We use it to map each factor into the product, and with the obtained basis of \( \H_n(\xSn;\ZZ) \) we identify the latter with \( \ZZ^{\x} \).

\begin{lemma}\label{realisation of automorphisms in Hn}
  The subgroup \( \Gamma\le\SL_{\x}(\ZZ) \) of matrices realised by orientation-preserving self homotopy equivalences of \( \xSn \) is \( \SL_{\x}(\ZZ) \) if \( n\in\{1,3,7\} \), and it consists of those matrices with exactly one odd entry in each row otherwise.
  The group realised by diffeomorphisms is the same one.
\end{lemma}
\begin{proof}
  The case \( \x=2 \) was covered in~\cite[Lemma 5]{Wall1962killingOdd}.
  The remaining cases of the group realised by diffeomorphisms were computed in~\cite{lucasDiffeomorphismsProductSpheres2002a}.
  The group realised by orientation-preserving homotopy equivalences was computed in~\cite{basuRealizingCongruenceSubgroups2016} and~\cite{BerglundZeman2024algebraicmodels}.
\end{proof}

Let \( p \) be an oriented \( \xSn \)-fibration over a path-connected space \( B \), together with a trivialisation \( \mathcal{H}_n(p)\cong\ZZ^{d}[n] \) of the local system on \( B \) given by the homology of the fibres, where \( \ZZ^d[n] \) is the standard representation of \( \Gamma \) put in degree \( n \).
Using this identification, we rewrite the transgression in the \( (n+1) \)st page of the Serre spectral sequence of \( p \) with rational coefficients as
\[
  d_{n+1}(p)\cl \QQ^{\x}[n]^{\vee}
  \lra
  \H^{n+1}(B;\QQ),
\]
where \( \vee \) denotes the degreewise dual.

\begin{definition}\label{def:Euler classes}
  The Euler classes \( e_1,\dotsc,e_{\x} \) of \( p \) are the images of the canonical basis for \(  \QQ^{\x}[n]^{\vee} \) under the differential \( d_{n+1}(p) \).
\end{definition}

Note that a \( \x \)-fold product of oriented \( \S^n \)-fibrations is an oriented \( \xSn \)-fibration.
The Euler classes of the product agree with the standard Euler classes of the oriented \( \S^n\)-fibrations, after pullback along the projections.

The classifying space for oriented \( \xSn \)-fibrations together with a trivialisation of the monodromy local system is the homotopy Torelli space \( \B\htor(\xSn) \), as defined in \cref{sec:class-fibr-1}.

Let \( \EulerRing \) be the polynomial algebra spanned by the universal Euler classes, where \( \xset=\{1,\dotsc,\x\} \).
Let \( \Gamma \) act on it as the symmetric powers of the dual of the standard representation.
Note that if \( \x \) is \( 2 \), we write a subscript \( x \) or \( y \) instead of 1 or 2.
\begin{proposition}\label{homotopy Torelli computation}
  The map
  \begin{equation}
    \label{eq:Euler Ring to htor}
    \EulerRing
    \lra
    \H^{*}(\B\htor(\xSn);\QQ)
  \end{equation}
  is a \( \Gamma \)-equivariant isomorphism.
\end{proposition}
\begin{proof}
  A proof is given in~\cite[Proposition 5.13]{BerglundZeman2024algebraicmodels}; cf. \cref{standard rep and its dual}.
  Equivariance can also be seen to hold because the spectral sequence used to define the Euler classes is one of \( \Gamma \)-modules.
\end{proof}

\begin{theorem}[{\cite[Theorem 5.15]{BerglundZeman2024algebraicmodels}}]\label{SnSn characteristic classes}
  The Serre spectral sequence with rational coefficients of \( \B\aut^+(\xSn)\ra\B\Gamma \) collapses at the second page, and it induces an isomorphism of graded \( \QQ \)-algebras
  \begin{equation}\label{eq:SnSn characteristic classes}
    \H^{*}(\B\aut^+(\xSn);\QQ)\cong \H^{*}(\Gamma;\EulerRing)
  \end{equation}
  where \( \Gamma \) acts as the symmetric powers of the dual of the standard representation.
\end{theorem}
\begin{remark}
  \label{vcd collapse}
  If \( \x \) is \( 2 \), then a different way to prove that there is an isomorphism~\eqref{eq:SnSn characteristic classes} is to observe that the virtual cohomological dimension of \( \Gamma \) is \( 1 \) (See \cref{vcd of Gamma}), which implies that the spectral sequence collapses.
  The algebra structures are trivial for degree reasons.
  To see this, observe that there are no invariants, and therefore the spectral sequence is concentrated in the first column.
  The second page of this spectral sequence is identified in \cref{homotopy Torelli computation}.
\end{remark}

\begin{remark}\label{standard rep and its dual}
  \cite[Theorem 5.15]{BerglundZeman2024algebraicmodels} is instead stated in terms of the symmetric powers \( \operatorname{Sym}^{\bullet} \QQ^{\x} \) of the standard representation.
  However, the pairs \( (\Gamma, \EulerRing) \) and \( (\Gamma,\operatorname{Sym}^{\bullet} \QQ^{\x}) \) are isomorphic in the category of pairs, on which group cohomology is a functor.
  More precisely, this isomorphism is given by the following pair \( (\iota,\ell) \).
  The group isomorphism \( \iota\cl\Gamma\cong\Gamma \) sends \( g \) to the transpose of \( g^{-1} \), which is indeed an element of \( \Gamma \) again.
  The map \( \ell \) is determined by sending the Euler classes to the canonical basis of \( \QQ^{\x} \) in the obvious way.
  Therefore, these two pairs are isomorphic, and so are their cohomologies.
\end{remark}

We refer to~\cite{BerglundZeman2024algebraicmodels} for more facts about these rings of characteristic classes, out of which we remark the following.
First, the ring of \( \Gamma \)-invariants of \( \EulerRing \) is trivial, and therefore all these characteristic classes are nilpotent.
Second, the Poincaré series of the rational cohomology of \( \B\aut^+(\SnSn) \) is given by
\[
  1+z^{2\ell+1}\frac{(1+z^{4\ell}-z^{6\ell})+z^{8\ell}}{(1-z^{4\ell})(1-z^{6\ell})}
\]
if \( n\in\{1,3,7\} \), and by
\[
  1+z\frac{(1+z^{2\ell}-z^{4\ell})+z^{6\ell}}{(1-z^{2\ell})(1-z^{4\ell})}
\]
otherwise, where \( \ell=n+1 \).
Finally, the reduced rational cohomology of \( \B\aut^+(\SnSn) \) is concentrated in odd degrees.
It follows that the ring structure on the rational cohomology is trivial, and that this space formal in the sense of rational homotopy theory.

\begin{remark}\label{rmk:vanishing of partial Wg1 to Wg}
  In the introduction, we claim that the map \(\B\aut_{\partial}(W_{1,1})\ra \B\aut^+(W_{1})\) vanishes on rational cohomology if \(2n\ge 6 \); recall that \( W_g=\#_{g}\SnSn\), that \( W_{g,1} \) is obtained by removing an open \( 2n \)-dimensional disc, and that the subscript \( \partial \) denotes that we restrict to automorphisms that fix the boundary pointwise.
  To verify this claim, observe that it is enough to do so on universal covers, because then one can use the corresponding spectral sequence.
  The rational cohomology ring of the universal cover of the codomain is generated in degree \( n+1 \), by \cref{homotopy Torelli computation}.
  Finally, that of the domain is trivial in degree \( n+1 \), which one can see from the description of its rational model in~\cite{BerglundMadsen2020}.

  We would like to add that there is an obvious generalisation of the Euler classes in \cref{def:Euler classes} for \( W_{g} \).
  However, it follows from \cite[Lemma 7.2]{KRW2020torelliCalc} that they vanish for \( g\ne1 \).
\end{remark}

\section{Generalised Dehn twists}\label{sec:gener-dehn-twists}
In this section, we prove results about the rational homotopy type of spaces of generalised Dehn twists.
We use the name generalised Dehn twists following~\cite{KreckSu2025mapping}.

Consider the topological monoid \( \hDehn \) with the pointwise product, together with the right action of the topological monoid \( \aut^+(\S^n) \) given by precomposition.
Similarly, consider the topological group \( \Dehn \) of smooth maps equipped with the smooth Whitney topology and the pointwise product, together with the right action of \( \SO(n+1) \) on it by precomposition.
We define maps of topological monoids
\begin{gather*}
  \hdehn\cl\hDehnH \lra \aut^{P}(\SnSn)
  \label{eq:homotopical dehn twist}
  \\
  \dehn\cl\DehnH \lra \Diff^{P}(\SnSn)
  \label{eq:smooth Dehn twists}
\end{gather*}
as follows.
The binary operation on the semidirect products for the given right actions \( \rho \) is given by
\[
  (f,g)\cdot(f',g')=(\rho(f,g')\cdot f',g\cdot g') .
\]
To a pair \( (f,g) \) we assign the map
\[
  \SnSn
  \lra
  \SnSn
  \qquad(x,y)\longmapsto (f(y)(x),g(y)).
\]
The superscript \( P \) denotes the submonoid consisting of those maps that on \( n \)th homology induce a map with matrix in \( P=\mathrm{U}\cap\Gamma \), where \( \mathrm{U} \) is the subgroup of \( \SL_2(\ZZ) \) generated by
\(
\begin{psmallmatrix}
  1&1\\0&1
\end{psmallmatrix}
\).

\begin{lemma}\label{Dehn has monodromy P}
  The group of matrices realised by \( \hdehn \) and \( \dehn \) is \( P \).
\end{lemma}
\begin{proof}
  Note that \( \aut^+(\S^n) \) and \( \SO(n+1) \) act trivially on homology.
  The fact that \( P \) is the group of matrices realised by \( \dehn \) and \( \hdehn \) is part of the proof of \cref{realisation of automorphisms in Hn} in the case \( d=2 \) given in~\cite[Lemma 5]{Wall1962killingOdd}.
\end{proof}

\begin{remark}\label{bundles and dehn twists}
  Denote by \( H \) either of the topological monoids \( \aut^+(\S^n) \) or \( \SO(n+1) \).
  Let \( \xi\cl\SnSn\ra\S^n \) be the projection to the second sphere.
  The topological monoid \( \map(\S^n,H)\rtimes H \) agrees with that consisting of bundle maps
  \[
    \begin{tikzcd}
      \SnSn
      \ar[r,"\varphi"]
      \ar[d]
      &\SnSn
      \ar[d]
      \\
      \S^n
      \ar[r,"g"]
      &\S^n
    \end{tikzcd}
  \]
  in which \( g \) is given by the action of \( H \), and \( \varphi \) is fibrewise given by the action of \( H \).
  Then
  \[
    \B(\map(\S^n,H)\rtimes H)
  \]
  classifies \( \S^n \)-fibrations \( E\ra B \), together with a fibre bundle \( \zeta \) over \( E \) with fibre \( \S^n \) and structure group \( H \) (an oriented \( \S^n \)-fibration \( \zeta \) if \( H=\aut^+(\S^n) \)), such that the restriction of \( \zeta \) to every fibre is trivial.
  See~\cite{Berglund2022characteristicClassesFamiliesBundles}.
\end{remark}

The remainder of this section is dedicated to providing rational differential graded (dg) Lie models for \( \hdehn \), and for the maps received from the domain of \( \dehn \).
We make use of the definitions in \cref{sec:geom-real-diff} involving dg Lie algebras and commutative dg algebras (cdga).

The rationalised homotopy groups of \( \aut^+(\S^n) \) and \( \SO(n+1) \) are the graded vector spaces
\begin{align*}\label{eq:lie models for monoids}
  &\aaa=\QQ\langle{\epsilon}\rangle\quad\text{and}
  \\
  &\sss=\QQ\langle{q_1,\dotsc,q_{\lfloor\frac{n}{2}\rfloor},\epsilon}\rangle,
\end{align*}
spanned by the following classes.
The class \( q_j \) has degree \( 4j-1 \) and is dual to the Pontryagin class \( p_j \) under the Hurewicz pairing.
Similarly, the class \( \epsilon \) has degree \( n \) and is dual to the Euler class.
We denote by \( \mathfrak{h} \) the rationalised homotopy groups of \( H \).

We consider monoids of automorphisms of two different spheres.
Let the spheres have as Sullivan model the exterior algebras \( \Lambda(x) \) and \( \Lambda(y) \) generated in degree \( n \), respectively.
When necessary, we denote the rationalised homotopy groups \( \mathfrak{h} \) with the corresponding subscripts \( x \) or \( y \).
Consider the action
\begin{equation}\label{eq:action of h}
  \mathfrak{h}\lra \Der\Lambda(y)
\end{equation}
defined by the projection \( \mathfrak{h}\ra\aaa \) followed by the action of \( \epsilon \) as \( \frac{\partial}{\partial y} \).

\begin{proposition}\label{action models}
  The action~\eqref{eq:action of h} models the action of \( H \) on \( \S^n \).
\end{proposition}
\begin{proof}
  By \cref{spaces/BG = G-spaces} and \cref{B and homotopy orbits}, it is enough to provide quasi-isomorphisms of cdga's \( (\varphi,\phi) \)
  \begin{equation}\label{eq:map of actions model}
    \begin{tikzcd}
      \CE(\mathfrak{h},\Lambda(y))
      \ar[r,"\phi"]
      &\Omega^{*}(\Sing\S^n_{\h H})
      \\
      \CE(\mathfrak{h})
      \ar[r,"\varphi"']
      \ar[u]
      &\Omega^{*}(\Sing\B H)
      \ar[u]
    \end{tikzcd}
  \end{equation}
  where \( \Sing \) is the singular simplicial set functor, and the polynomial differential forms \( \Omega^{*} \) of a simplicial set \( K \) is the cdga \( \Hom_{\mathsf{sSet}}(K,\Omega_{\bullet}) \).
  These are the adjoint functors of \( |-| \) and \( \langle{-}\rangle \), respectively.

  From the proof of~\cite[Proposition 3.7]{Berglund2022characteristicClassesFamiliesBundles}
  we extract a zig-zag of weak homotopy equivalences
  \[
    \langle\CE(\aaa)\rangle
    \sim
    \B|\exp_{\bullet}\aaa|
    \sim
    \B\aut^+(\S^n).
  \]
  The functor \( \Omega^{*}\Sing \) sends weak homotopy equivalences to quasi-isomorphisms.
  We obtain a map \( \varphi\cl\CE(\aaa)\ra\Omega^{*}(\Sing\B H) \) by using the cofibrancy of \( \CE(\aaa) \) to lift maps along quasi-isomorphisms.
  Observe that \( \CE(\mathfrak{h}) \) is free as an algebra, and the given generators make clear how to define an isomorphism to \( \H^{*}(\B H;\QQ) \).
  We extend \( \varphi \) along the map \( \CE(\aaa)\ra \CE(\mathfrak{h}) \) induced by the projection \( \mathfrak{h}\ra\mathfrak{a} \).
  We do this by picking cocycle representatives for the given generators of \( \H^{*}(\B H;\QQ) \) while keeping the choice for the Euler class that defines \( \varphi \).
  To define \( \phi \), observe that \( \CE(\mathfrak{h},\Lambda(y)) \)
  is isomorphic to the free graded-commutative algebra \( \CE(\mathfrak{h})\otimes\Lambda(y)\) with differential \( d(y)=e \).
  The Euler class \( e \) vanishes in \( \H^{*}(\S^n_{\h H};\QQ)=0 \).
  By sending \( y \) to a witness of this fact in \( \Omega^{*}(\Sing\S^n_{\h H}) \) we define the map \( \phi \) which completes the diagram~\eqref{eq:map of actions model}.
\end{proof}

Consider the dg Lie algebra
\[
    (\Lambda(y)\otimes\mathfrak{h}_{x})\langle{0}\rangle \rtimes\mathfrak{h}_{y}
\]
where the action of \( \epsilon_y \) is given by \(\epsilon_y\cdot( a\otimes p)=(\frac{\partial}{\partial y}a)\otimes p \), and \( q_{j,y} \) acts trivially for all \( j \).

\begin{proposition}\label{model for Dehn spaces}
  The map
  \begin{equation}\label{eq:B of map of Dehn spaces}
    \B(\DehnH)
    \lra
    \B(\hDehnH)
  \end{equation}
  admits a rational section.
  Its rational model is given by the map
  \begin{equation}\label{eq:model of map of Dehn spaces}
    (\Lambda(y)\otimes\sx)\langle{0}\rangle \rtimes\sy
    \lra
    (\Lambda(y)\otimes \ax)\langle{0}\rangle \rtimes \ay
  \end{equation}
  induced by the projection \( \sss\ra\aaa \), and the section is induced by the inclusion \( \aaa\ra\sss \).
\end{proposition}
\begin{proof}
  The inclusion \( \mathrm{C}^{\infty}(\S^n,\SO(n+1)) \ra \map(\S^n,\SO(n+1)) \) into the space of continuous maps equipped with the compact-open topology is a weak homotopy equivalence by the Whitney Approximation Theorem.
  See e.g.~\cite[Theorem 6.26]{Lee2013smoothmanifolds}.

  We apply~\cite[Theorem 3.8]{Berglund2022characteristicClassesFamiliesBundles} to provide the rational dg Lie algebra model for~\eqref{eq:B of map of Dehn spaces}.
  The topological monoid \( \map(\S^n,H)\rtimes H \) agrees with the topological monoid denoted by \( \aut_{H}(\xi) \) in the reference when we let \( \xi \) be the trivial \( \S^n \)-bundle given by the projection to the second sphere.
  Since the given action of \( \mathfrak{h} \) on \( \Lambda(y) \) models the action of \( H \) on \( \S^n \) by \cref{action models}, the result follows from loc.\ cit.

  We spell out the details of the proof using the functor \( \exp_{\bullet} \) instead of the functor \( \MCb \) used in the reference because we use the details in the proof of \cref{homotopical Dehn twist model}.
  We provide a natural weak homotopy equivalence of simplicial groups
  \begin{equation}\label{eq:Dehn space exp model}
    \exp_{\bullet}((\Lambda(y)\otimes \mathfrak{h}_{x})\langle{0}\rangle \rtimes \mathfrak{h}_{y})
    \lra
    \map(\langle{\Lambda(y)}\rangle,\exp_{\bullet}\mathfrak{h}_{x})\rtimes \exp_{\bullet}\mathfrak{h}_{y}
  \end{equation}
  where \( \exp_{\bullet}\mathfrak{h}_{y} \) acts on the simplicial set of maps by precomposition by the action defined in \cref{sec:geom-real-diff}.
  (The geometric realisation of the simplicial set of maps between two simplicial sets is naturally weak homotopy equivalent to the topological space of maps between their geometric realisations.)

  The natural map
  \[
    \exp_{\bullet}((\Lambda(y)\otimes \mathfrak{h}_{x})\langle{0}\rangle) \rtimes\exp_{\bullet}\mathfrak{h}_{y}
    \lra
    \exp_{\bullet}((\Lambda(y)\otimes \mathfrak{h}_{x})\langle{0}\rangle \rtimes\mathfrak{h}_{y})
  \]
  that in each simplicial degree sends a pair \( (\alpha,\beta) \) to their Baker--Campbell--Hausdorff product \( \alpha*\beta \) is an isomorphism by \cref{exp and conjugation}.
  To obtain~\eqref{eq:Dehn space exp model} we define a map
  \begin{equation}
    \label{eq:exp Dehn}
    \exp_{\bullet}((\Lambda(y)\otimes \mathfrak{h}_{x})\langle{0}\rangle)
    \lra
    \map(\langle{\Lambda(y)}\rangle,\exp_{\bullet}\mathfrak{h}_{x})
  \end{equation}
  that is equivariant for the actions of \( \exp_{\bullet}\mathfrak{h}_{y} \).
  Let the adjoint of~\eqref{eq:exp Dehn} be the map that sends a pair \( (b\otimes\theta,q) \) of simplices to \( q(b)\otimes \theta \).
  Then~\eqref{eq:exp Dehn} is a weak homotopy equivalence by~\cite[Theorem 6.6]{Berglund2015mappingLinfinity} applied to \( L=s^{-1}\mathfrak{h}_{x} \) (that is, \( \mathfrak{h}_x \) with the grading decreased by \( 1 \)) and \( X=\langle{\Lambda(y)}\rangle \), combined with the following two facts.
  The counit \( \Lambda(y)\ra\Omega\langle{\Lambda(y)}\rangle \) is a quasi-isomorphism.
  There is a natural isomorphism \( \exp\cong\MC s^{-1} \) of functors to the category of sets.
\end{proof}

Consider the dg Lie algebra of derivations \( \Der\Lambda(x,y) \) of the Sullivan model of \( \SnSn \).
Let \( \ppp^{*} \) be the Lie subalgebra of degree 0 derivations spanned by \( y\frac{\partial}{\partial x} \).
Let \( \ppp\le\mathfrak{sl}_2 \) be the Lie subalgebra corresponding to \( P\le \SL_2(\QQ) \).
That is, the Lie algebra of strict upper triangular matrices.

\begin{proposition}\label{model for BautP}
  The map \( \B\aut^P(\SnSn)\ra \B P \) has as a dg Lie algebra model the map
  \begin{equation}\label{eq:model for BautP over BP}
    \Der_{\ppp^{*}}\Lambda(x,y)=
    \QQ\left\langle{\frac{\partial}{\partial x}, \frac{\partial}{\partial y}}\right\rangle \rtimes \QQ\left\langle{y\frac{\partial}{\partial x}}\right\rangle
    \lra
    \ppp
  \end{equation}
  that in degree 0 sends \( -y\frac{\partial}{\partial x} \) to
  \(
  \begin{psmallmatrix}
    0&1\\0&0
  \end{psmallmatrix}
  \).
\end{proposition}
\begin{proof}
  Let \( \mathrm{E} \) be the group of homotopy classes of self homotopy equivalences of \( \SnSn \), and let \( \mathrm{E}(\QQ) \) be that of its rationalisation. Let \( \rho\cl \mathrm{E} \ra \mathrm{E}(\QQ) \) be the homomorphism induced by rationalisation.
  By~\cite[Corollary 4.13]{BerglundZeman2024algebraicmodels}, a rational dg Lie model for \( \B\aut^P(\SnSn)\ra\B P \) is the pullback of the map \( \Der\Lambda(x,y)\langle{0}\rangle\ra \H_{0}\Der\Lambda(x,y) \) along the following map \( \upsilon \).
  Let \( \uuu \) be the Lie algebra of the minimal unipotent subgroup \( U \) of \( \mathrm{E}(\QQ) \) that contains the image under \( \rho \) of the fundamental group \( G \) of \( \B\aut^P(\SnSn) \).
  (Note that \( G \) acts nilpotently on the homology of \( \SnSn \)).
  The isomorphism \( \sigma \) between \( \mathrm{E}(\QQ) \) and the group \( \mathcal{A}\mathrm{ut}\,\Lambda(x,y) \) of homotopy classes of self-automorphisms of the Sullivan model gives the monomorphism \( \upsilon \cl \uuu\ra \H_0\Der\Lambda(x,y) \) on Lie algebras.
  All in all, a rational model is the pullback \( \Der_{\upsilon(\uuu)}\Lambda(x,y) \) along \( \upsilon(\uuu) \) together with its map to \( \uuu \).

  To determine \( \upsilon \), note that the action of \( \mathrm{E}(\QQ) \) on degree \( n \) rational homology, and our fixed basis for \( \H_n(\SnSn;\QQ) \), give an isomorphism \( \mathrm{E}(\QQ)\cong \GL_2(\QQ) \).
  Under this identification, the image of \( G \) under \( \rho \) is \( P \).
  The minimal unipotent subgroup of \( \GL_2(\QQ) \) that contains \( P \) is the subgroup of upper triangular matrices.
  Its Lie algebra \( \ppp \) is that of strict upper triangular matrices.
  The action of \( \mathrm{E}(\QQ) \) on degree \( n \) rational cohomology implements the isomorphism \( \sigma \) between \( \mathrm{E}(\QQ) \) and \( \mathcal{A}\mathrm{ut}\,\Lambda(x,y)=\GL(\QQ\langle{x,y}\rangle) \).
  Therefore, \( \upsilon\cl \ppp\ra\mathfrak{gl}(\QQ\langle{x,y}\rangle) \) sends a matrix
  \(
  \begin{psmallmatrix}
    0&\lambda\\0&0
  \end{psmallmatrix}
  \)
  to \( -\lambda y\frac{\partial}{\partial x} \) and \( \upsilon(\ppp)=\ppp^{*} \).
  To conclude, observe that in positive homological degrees \( \Der\Lambda(x,y) \) consists of the rational vector space \( \QQ\langle{\frac{\partial}{\partial x}, \frac{\partial}{\partial y}}\rangle \) of derivations of degree \( n \).
\end{proof}

Consider the canonical maps
\[
  \Der\Lambda(x)
  \lra
  \Der\Lambda(x,y)
  \qquad
  \Der\Lambda(y)
  \lra
  \Der\Lambda(x,y).
\]
We use the same notation for a derivation and its image under these maps.
The map
\[
  \Lambda(y)\otimes \Der\Lambda(x)
  \lra
  \Der\Lambda(x,y)
\]
that sends \( b\otimes\theta \) to the derivation \( (b\theta)(z)=b\theta(z) \) is \( \Der\Lambda(y) \)-equivariant, the action on the target being the adjoint representation.
Therefore we get a maps of dg Lie algebras
\begin{equation}\label{eq:homotopical Dehn twist model}
  (\Lambda(y)\otimes \aaa_x)\rtimes\aaa_{y}
  \lra
  \Lambda(y)\otimes\Der\Lambda(x)\rtimes\Der\Lambda(y)
  \lra
  \Der\Lambda(x,y)
\end{equation}
when combined with~\eqref{eq:action of h}.

\begin{proposition}\label{homotopical Dehn twist model}
  The map~\eqref{eq:homotopical Dehn twist model} is a dg Lie algebra model for the map
  \[
    \B(\hDehnH) \lra \B\aut^P(\SnSn) .
  \]
\end{proposition}
\begin{proof}
  The maps of simplicial monoids
  \begin{gather*}
    \exp_{\bullet}\ay
    \lra
    \aut^+\langle{\Lambda(y)}\rangle
    \\
    \exp_{\bullet}\Der_{\ppp^{*}}(\Lambda(x,y))
    \lra
    \aut^{P\otimes\QQ}\langle{\Lambda(x,y)}\rangle
  \end{gather*}
  induced by the actions by derivations of cdga's are weak homotopy equivalences by the proof of~\cite[Proposition 3.13]{BerglundZeman2024algebraicmodels}.
  Together with~\eqref{eq:Dehn space exp model}, these provide the vertical weak homotopy equivalences in the diagram
  \begin{equation}\label{eq:map modelling}
    \begin{tikzcd}
      \exp_{\bullet}((\Lambda(y)\otimes \ax)\langle{0}\rangle \rtimes \ay)
      \ar[r]
      \ar[d]
      &\exp_{\bullet}\Der_{\ppp^{*}}(\Lambda(x,y))
      \ar[d]
      \\
      \map(\langle{\Lambda(y)}\rangle,\aut^+\langle{\Lambda(x)}\rangle)\rtimes\aut^+\langle{\Lambda(y)}\rangle
      \ar[r]
      &\aut^{P\otimes\QQ}\langle{\Lambda(x,y)}\rangle
    \end{tikzcd}
  \end{equation}
  the commutativity of which we have to prove. The top map is \( \exp_{\bullet} \) applied to~\eqref{eq:homotopical Dehn twist model}, the bottom map is defined as \( \hdehn \) is,
  using the natural isomorphism
  \begin{equation}
    \label{eq:realisation and product}
    \langle{\Lambda(x)}\rangle \times \langle{\Lambda(y)}\rangle \cong \langle{\Lambda(x,y)}\rangle .
  \end{equation}
  A zig-zag between the bottom map and \( \hdehn \) consisting of maps that induce an isomorphism on all rational homotopy and homology groups of their classifying spaces can be given as in~\cite[Lemma 3.1]{Berglund2020autfiberbundles} and~\cite[Proposition 3.10]{BerglundZeman2024algebraicmodels}.
  On the source of the bottom map of~\eqref{eq:map modelling}, this zig-zag is induced by \( (\varphi,\phi) \), defined in the proof of \cref{action models}.

  To conclude, we check that the diagram~\eqref{eq:map modelling} commutes.
  Under the natural isomorphism~\eqref{eq:realisation and product} a pair of simplices \( (p,q) \) corresponds to their coproduct \( p\otimes q \).
  We also write \( p \) and \( q \) for the maps \( p\otimes 1 \) and \( 1\otimes q \), respectively, where \( 1 \) denotes the constant map to the unit.
  Therefore, we use the notation \( pq \) for \( p\otimes q \).
  The adjoint of the bottom composite in~\eqref{eq:map modelling}
  sends \( (b\otimes\theta)*\delta \) and a pair of simplices \( (p,q) \), to the pair
  \[
    ( e^{q(b)\otimes\theta}(p), e^{\delta}(q) ) .
  \]
  The adjoint of the top composite sends \( (b\otimes\theta)*\delta \) and \( pq \), to
  \begin{align*}
    e^{(b\theta)*\delta}(pq)
    &=(e^{b\theta}e^{\delta})(pq)
    \\
    &=e^{b\theta}(pq)e^{\delta}(pq)
      \\
    &=e^{q(b)\theta}(p)e^{\delta}(q).
      \qedhere
  \end{align*}
\end{proof}

\begin{proposition}\label{monodromy P action on Dehn space}
  The map
  \begin{equation}
    \label{eq:smooth Dehn fibration over BP}
    \B(\DehnH)
    \lra
    \B P
  \end{equation}
  has as a dg Lie model the map
  \begin{equation}
    \label{eq:model smooth Dehn fibration over BP}
    (\Lambda(y)\otimes \sx)\langle{0}\rangle\rtimes \sy
    \lra
    \ppp
  \end{equation}
  that sends \( -y\otimes\epsilon_{x} \) to
  \(
  \begin{psmallmatrix}
    0&1\\0&0
  \end{psmallmatrix}
  \)
  and all other basis elements to zero.
  The monodromy action of~\eqref{eq:smooth Dehn fibration over BP} is modelled by the action of \( \ppp \) on the kernel of~\eqref{eq:model smooth Dehn fibration over BP} given by the standard representation on the summand \( \ax\oplus\ay \), and the trivial action on the complementary summand.
\end{proposition}
\begin{proof}
  The map~\eqref{eq:model smooth Dehn fibration over BP} is a model because it is the composite of~\eqref{eq:model of map of Dehn spaces}, \eqref{eq:homotopical Dehn twist model} and \eqref{eq:model for BautP over BP}.
  The statement about the monodromy action follows from \cref{B and homotopy orbits} once we identify \( \mathfrak{d}=(\Lambda(y)\otimes \sx)\langle{0}\rangle\rtimes \sy \) as the semidirect product of the action given in the statement.
  Let \( \operatorname{ad}\cl\mathfrak{d} \ra \Der(\mathfrak{k}) \) be the restriction of the adjoint representation of \( \mathfrak{d} \) to the kernel \( \mathfrak{k} \) of~\eqref{eq:model smooth Dehn fibration over BP}.
  Let \( \sigma \) be the section of~\eqref{eq:model smooth Dehn fibration over BP} that sends
  \(
  \begin{psmallmatrix}
    0&1\\0&0
  \end{psmallmatrix}
  \)
  to \( -y\otimes \epsilon_{x} \).
  The action \( \operatorname{ad}\circ\operatorname{\sigma}\cl\ppp\ra\Der(\mathfrak{k}) \) is given by
  \[
    \begin{psmallmatrix}
      0&1\\0&0
    \end{psmallmatrix}
    \cdot a
    =[-y\frac{\partial}{\partial x},a]
    =[a,y\frac{\partial}{\partial x}]
    =\left\{
    \begin{array}{ll}
      \epsilon_{x} &\text{if }a=\epsilon_y
      \\
      0&\text{if }a\ne\epsilon_y
    \end{array}
  \right.
  .
  \qedhere
  \]
\end{proof}

\section{An injection of characteristic classes}\label{sec:injection of characteristic classes}

In this section, we prove \cref{intro:injection of characteristic classes}.
We derive it from a result that holds for the \( d \)-fold product \( \xSn \).

Let \( \Tor(\xSn) \) be the topological subgroup of \( \Diff^+(\xSn) \) formed by those path components consisting of diffeomorphisms that induce the identity on homology; cf. \cref{sec:class-fibr-1}.
We consider the map
\begin{equation*}
  \label{eq:Euler map to Tor}
  \EulerMap
  \cl
  \EulerRing
  \lra
  \H^{*}(\B\Tor(\xSn);\QQ)
\end{equation*}
from the polynomial algebra spanned by the universal Euler classes of \cref{sec:euler-classes-SnSn-fibrations}.
By \cref{homotopy Torelli computation}, it is recovered as the map induced on rational cohomology by
\[
  \B\Tor(\xSn)
  \lra
  \B\htor(\xSn).
\]

The map of \( \Gamma \)-representations \( \EulerMap \) is injective.
A retraction of \( \EulerMap \) as vector spaces can be defined as follows.
Consider the map
\begin{equation*}
  \label{eq:product of vector bundles}
  \FibProd
  \cl
  \times_{d}\B\SO(n+1)
  \lra
  \B\Tor(\xSn)
\end{equation*}
which on topological groups is given by sending a tuple of maps to their product.
We write the rational cohomology of \( \times_{\x}\B\SO(n+1) \) as
\begin{equation}
  \label{eq:Q cohomology of SO}
  \EulerPontryaginRing,
\end{equation}
where \( p_{j,i} \) is the \( j \)th Pontryagin class in the \( i \)th factor of \( \times_{\x}\B\SO(n+1) \).
It is a polynomial algebra in the given variables.
The restriction to \( \EulerRing \) of the map induced by \( \FibProd \) on rational cohomology is the inclusion, of which the projection defines a retraction.
We obtain a retraction of \( \EulerMap \)
\begin{equation*}
  \label{eq:retraction to Euler classes}
  \Retr
  \cl
  \H^{*}(\B\Tor(\xSn);\QQ)
  \lra
  \EulerRing.
\end{equation*}

\begin{theorem}
  \label{Torelli equivariant splitting}
  The retraction \( \Retr \) is \( \Gamma \)-equivariant.
  In particular,
  \[
    \EulerMap
    \cl
    \EulerRing
    \lra
    \H^{*}(\B\Tor(\xSn);\QQ)
  \]
  is split injective as a map of \( \Gamma \)-representations.
\end{theorem}

\begin{corollary}\label{injection of fibration classes}
    The map induced by
          \begin{equation}\label{eq:split injection of fibration classes}
            \B\Diff^{+}(\SnSn)
            \lra
            \B\aut^{+}(\SnSn)
          \end{equation}
          on rational cohomology is injective.
\end{corollary}
\begin{proof}
  The map~\eqref{eq:split injection of fibration classes} over \( \B \Gamma \) induces a map between Serre spectral sequences.
  On the second page it takes the form
  \begin{equation}\label{eq:second page Torelli}
    \H^*(\Gamma;\H^{*}(\B\htor(\SnSn);\QQ))
    \lra
    \H^*(\Gamma;\H^{*}(\B\Tor(\SnSn);\QQ)),
  \end{equation}
  and the map of \( \Gamma \)-representations inducing it can be identified with the map
  \begin{equation*}
    \label{eq:3}
    \EulerMap\cl
    \EulerRing \lra\H^{*}(\B\Tor(\xSn);\QQ) ,
  \end{equation*}
  by \cref{SnSn characteristic classes}.
  \( \EulerMap \) is split injective as a map of \( \Gamma \)-representations, by \cref{Torelli equivariant splitting}.
  Both spectral sequences collapse at the second page because the virtual cohomological dimension of \( \Gamma \) is \( 1 \) (see \cref{vcd collapse}).
\end{proof}

\begin{remark}
  If \( \x>2 \), then \( \Retr \) also induces a retraction of \eqref{eq:second page Torelli} in the category of bigraded vector spaces.
  We do not know if the splitting is compatible with the differentials.
  In particular, we do not know if \cref{Torelli equivariant splitting} can be lifted to a space-level statement, rationally.
  The latter would also imply that the splitting on the rational cohomology of \eqref{eq:split injection of fibration classes} is compatible with the ring structures.
  These statements do hold for \( \times_{\x}\S^{1} \); cf.\ \eqref{eq:chi for S1}.
\end{remark}

In the remainder of this section, we prove \cref{Torelli equivariant splitting}.

Let \( \T \) be the subgroup of \( \Gamma \) generated by the subset
\begin{gather*}
  \{\T_{ij}\mid 1\le i,j\le \x\}\quad\text{if }n\in\{1,3,7\},
  \\
  \{\T_{ij}^{2}\mid 1\le i,j\le \x\}\quad\text{if }n\notin\{1,3,7\},
\end{gather*}
where \( \T_{ij} \) has \( 1 \) in each diagonal entry, \( 1 \) in the entry \( (i,j) \), and \( 0 \) elsewhere.

Let \( \Sigma \) be the group of signed permutation matrices with determinant equal to \( 1 \) and size \( d \).
That is, the kernel
\[
  1
  \lra
  \Sigma
  \lra
  \mathrm{O}_{\x}(\ZZ)
  \lra[\mathrm{det}]
  \{\pm 1\}
  \lra
  1
\]
where \( \mathrm{O}_\x(\ZZ) \) are the orthogonal matrices with integer entries.
It is also known as the hyperoctahedral group.
It is a subgroup of \( \Gamma \).

\begin{lemma}\label{transvections and signed permutations generate}
  \( \Gamma \) is generated by \( \T \) and \( \Sigma \).
\end{lemma}
\begin{proof}
  If \( n \) is 1,3 or 7, then \( \T \) is \( \SL_{\x}(\ZZ) \).
  Otherwise, \( \Gamma \) fits in a short exact sequence
  \[
    1
    \lra
    \SL_{\x}(\ZZ,2)
    \lra
    \Gamma
    \lra
    \Sigma_{\x}
    \lra
    1
  \]
  where \( \Sigma_{\x} \) is the subgroup of permutation matrices in \( \SL_{\x}(\ZZ/2\ZZ) \)
  (it is isomorphic to the symmetric group of degree \( \x \)),
  and the kernel is the principal congruence subgroup of level 2.
  The subgroup \( \Sigma \) surjects onto \( \Sigma_{\x} \).
  The group \( \SL_{\x}(\ZZ,2) \) is generated by \( \T \) and \( \SL_{\x}(\ZZ,2)\cap\Sigma \).
  Indeed, see~\cite{basuRealizingCongruenceSubgroups2016}.
\end{proof}

\paragraph{Transvections}
Out of the vector spaces involved in the definition of the retraction \( \Retr \), the only one that is not a priori equipped with an action of \( \T \) is
\[
  \EulerPontryaginRing.
\]
We define the action by taking symmetric powers of the following action on the graded vector space
\[
    \QQ\langle{e_{i},p_{1,i},\dotsc,p_{\lfloor\frac{n}{2}\rfloor,i}\mid i\in\xset}\rangle
\]
spanned by the Euler classes and Pontryagin classes.
Let \( \T \) act on \( \QQ\langle{e_{i}\mid i\in\xset}\rangle \) as the dual of the standard representation, and trivially on \( \QQ\langle{p_{1,i},\dotsc,p_{\lfloor\frac{n}{2}\rfloor,i}\mid i\in\xset}\rangle \).

\begin{proposition}\label{E equivariance of retraction}
  The map induced by \( \FibProd \)
  \[
    \H^{*}(\B\Tor(\xSn);\QQ)
    \lra
    \EulerPontryaginRing
  \]
  is \( \T \)-equivariant.
  In particular, the retraction \( \Retr \) is \( \T \)-equivariant.
\end{proposition}
\begin{proof}
  It is enough to cover the case  \( \x=2 \).
  We prove \( P \)-equivariance, where \( P \) is the intersection of \( \T \) with the subgroup of upper triangular matrices in \( \SL_2(\ZZ) \).
  The case of lower triangular matrices is proven analogously.
  The map of topological groups
  \[
    \SO(n+1)
    \lra
    \Dehn
  \]
  that sends a map to the constant map at it yields a map
  \[
    \widetilde{\FibProd}
    \cl
    \B\SO(n+1)^{\times 2}
    \lra
    \B\Tor\dehn
  \]
  to the homotopy fibre \( \B \Tor\dehn \) of
  \[
    \B(\DehnH)
    \lra
    \B P.
  \]
  The map \( \dehn \) induces a map
  \begin{equation}
    \label{eq:BTorDP to BTor}
    \dehn\cl
    \B\Tor \dehn
    \lra
    \B\Tor(\SnSn)
  \end{equation}
  which is equivariant with respect to the monodromy actions of \( P \).
  This provides a factorisation of \( \FibProd \) as the composite of \( \widetilde{\FibProd} \) and \eqref{eq:BTorDP to BTor}.
  It remains to check that \( \widetilde{\FibProd} \) induces a \( P \)-equivariant map on rational cohomology.
  It holds at the level of dg Lie models by \cref{monodromy P action on Dehn space}.
\end{proof}

\paragraph{Signed permutation matrices}
Again, out of the vector spaces involved in the definition of the retraction \( \Retr \),
\[
  \EulerPontryaginRing
\]
remains to be equipped with an action of \( \Sigma \).
In this case, an action can be defined on the space level, or equivalently, on the topological group \( \times_{\x}\SO(n+1) \).
This can be done as follows.

Let \( \mathrm{r}\cl \S^n\ra\S^n \) be a reflection, say,
\[
  \mathrm{r}(x_0,x_{1},\dotsc,x_{n})=(-x_0,x_1,\dotsc,x_n)
\]
in \( \RR^{n+1} \).
The group \( \mathrm{O}_d(\ZZ) \) is isomorphic to the group of bijections \( \sigma \) of \( \{\pm1,\dotsc, \pm n\} \) such \( \sigma(-i)=-\sigma(i) \).
The orthogonal matrix \( A_{\sigma} \) corresponding to \( \sigma \) has \( (|\sigma(i)|,i) \)-entry the number \( (-1)^{s(i)} \) where
\[
  s(i)
  =
  \left\{
    \begin{array}{ll}
      1&\text{if }\sigma(i)<0
      \\
      0&\text{if }\sigma(i)>0.
    \end{array}
  \right.
\]
Let \( \sigma\in \Sigma \) act on the topological group \( \times_{\x}\SO(n+1) \) by
\[
  (f_1.\dotsc,f_{\x})
  \longmapsto
  (\mathrm{c}_{\mathrm{r}}^{s(1)}f_{|\sigma(1)|},\dotsc, \mathrm{c}_{\mathrm{r}}^{s(\x)}f_{|\sigma(\x)|}),
\]
where \( \mathrm{c}_{\mathrm{r}} \) denotes the conjugation by \( \mathrm{r} \) in \( \SO(n+1) \).

\begin{remark}\label{rmk:Sigma on BSO}
  The effect of conjugating by \( \mathrm{r} \) on the rational cohomology of \( \B\SO(n+1) \) is trivial on the Pontryagin classes.
  Indeed, the Pontryagin classes are defined in the rational cohomology of \( \B\operatorname{O}(n+1) \), on which conjugation by \( \mathrm{r} \) acts trivially (see the proof of \cref{conjugation=monodromy}).
  Conjugation by \( \mathrm{r} \) acts as \( -1 \) on the Euler class, because we have the commutative diagram
  \[
    \begin{tikzcd}
      \SO(n+1)
      \ar[r,"\mathrm{c}_{\mathrm{r}}"]
      \ar[d]
      &\SO(n+1)
      \ar[d]
      \\
      \S^n
      \ar[r,"\mathrm{r}"]
      &\S^n,
    \end{tikzcd}
  \]
  where the vertical map is given by evaluation at a point in the fixed points of \( \mathrm{r} \), say, \( (0,\dotsc,0,1) \).
  Finally, observe that \( \mathrm{r} \) has degree \( -1 \).

  Therefore, the \( \Sigma \) representation \( \EulerPontryaginRing \) is obtained as the symmetric powers of the \( \Sigma \)-representation
  \[
    \QQ\langle{e_{i}\mid i\in\xset}\rangle
    \oplus
    \QQ\langle{p_{1,i}\mid i\in\xset}\rangle
    \oplus
    \dotsb
    \oplus
    \QQ\langle{p_{\lfloor\frac{n}{2}\rfloor,i}\mid i\in\xset}\rangle
  \]
  that is a direct sum of the following \( d \)-dimensional \( \Sigma \)-representations.
  The action of \( \sigma\in \Sigma \) on \( \QQ\langle{e_i\mid i\in\xset}\rangle \) is given by
  \[
    \H^{*}(\sigma^{-1})(e_{j})=(-1)^{s(j)}e_{|\sigma(j)|},
  \]
  with matrix \( A_{\sigma} \), defined above.
  For each \( j \), \( \Sigma \) acts on \( \QQ\langle{p_{j,i}\mid i\in\xset}\rangle \) through \( \Sigma\ra\Sigma_{\x} \) as the permutation representation of the symmetric group \( \Sigma_d \).
\end{remark}

\begin{lemma}\label{Sigma equivariance of retraction}
  The map induced by \( \FibProd \)
  \[
    \H^{*}(\B\Tor(\xSn);\QQ)
    \lra
    \EulerPontryaginRing
  \]
  is \( \Sigma \)-equivariant.
  In particular, the retraction \( \Retr \) is \( \Sigma \)-equivariant.
\end{lemma}
\begin{proof}
  Suppose that the map \( \Diff^+(\xSn) \ra \Gamma \) that applies degree \( n \) homology admits a section on \( \Sigma \).
  Then \( \Sigma \) acts on \( \Tor(\xSn) \) by conjugation.
  The induced action on the classifying space agrees up to homotopy with the monodromy action, by \cref{conjugation=monodromy}.

  We define such a section by assigning to a signed permutation \( \sigma \) the self-diffeomorphism of \( \xSn \) given by
  \[
    (x_1,\dotsc,x_{d})
    \longmapsto
    (\mathrm{r}^{s(1)}x_{|\sigma(1)|},\dotsc,\mathrm{r}^{s(d)}x_{|\sigma(d)|}).
  \]
  Then the map
  \[
    \FibProd\cl\times_{d}\SO(n+1)\ra \Tor(\xSn)
  \]
  is \( \Sigma \)-equivariant.
  To conclude, we extract from \cref{rmk:Sigma on BSO} that the inclusion and the projection,
  \[
    \begin{tikzcd}
      \EulerRing
      \ar[r,shift left]
      &
      \EulerPontryaginRing,
      \ar[l,shift left]
    \end{tikzcd}
  \]
  are equivariant.
\end{proof}

\begin{proof}[Proof of \cref{Torelli equivariant splitting}]
  The retraction \( \Retr \) is \( \T \)-equivariant by \cref{E equivariance of retraction}, and it is \( \Sigma \)-equivariant by \cref{Sigma equivariance of retraction}.
  We deduce \( \Gamma \)-equivariance because the group \( \Gamma \) is generated by its subgroups \( \T \) and \( \Sigma \) by \cref{transvections and signed permutations generate}.
\end{proof}

We need the following variant of \cref{injection of fibration classes} in \cref{sec:decomp-bund}.
(See \cref{sec:class-fibr-1} for notation.)

\begin{theorem}\label{injection of characteristic classes; monodromy in G}
  Let \( K\le\Gamma \) be a subgroup.
  Then the map induced by
  \begin{equation}
    \B\Diff^{K}(\SnSn)
    \lra
    \B\aut^{K}(\SnSn)
  \end{equation}
  on rational cohomology is injective.
\end{theorem}
\begin{proof}
  Since the virtual cohomological dimension of \( \Gamma \) is 1, that of \( K \) is at most 1.
  Therefore, the cohomological Serre spectral sequence of \( \B\aut^{K}(\SnSn)\ra\B K \) with rational coefficients collapses, because there is no room for non-trivial differentials.
  To finish, one can argue as in the proof of \cref{injection of fibration classes}.
\end{proof}

The following result is not used in this article, but it might be of interest.
\begin{proposition}
  Let \( n \) be \( 1 \), \( 3 \), or \( 7 \).
  \begin{equation}
    \label{eq:BTor to Euler and Pontryagin}
    \H^{*}(\B\Tor(\xSn);\QQ)
    \lra
    \PontryaginRing
  \end{equation}
  induced by \( \FibProd \) followed by the projection has image in the \( \Sigma \)-invariants.
\end{proposition}
\begin{proof}
  We show in \cref{Sigma equivariance of retraction} that \eqref{eq:BTor to Euler and Pontryagin} is \( \Sigma \)-equivariant.
  The assumption on \( n \) implies that \( \T \) equals \( \Gamma \); in particular, \( \T \) contains \( \Sigma \).
  Then it follows from \cref{E equivariance of retraction} that \eqref{eq:BTor to Euler and Pontryagin} is also equivariant for the trivial \( \T \)-action on the target.
\end{proof}

\paragraph{Homotopy spheres}
A homotopy \( n \)-sphere is a \( n \)-dimensional smooth manifold that is homotopy equivalent to the \( n \)-sphere \( \S^n \).
Let \( X_{\x} \) be a product of homotopy \( n \)-spheres.
We note that \( \Tor(X_{\x})\ra \htor(X_{\x}) \) is injective on rational cohomology.
This holds because \( \htor(X_{\x}) \) is rationally equivalent to \( \times_{\x}\aut_{\id}(\S^n) \) by \cref{homotopy Torelli computation}, and \( \Diff_{\id}(\Sigma^n) \) is rationally equivalent to \( \Diff_{\id}(\S^n) \) over \( \aut_{\id}(\S^n) \) by Dwyer--Szczarba~\cite{dwyerHomotopyTypeDiffeomorphism1983}.

One can ask whether
\begin{equation}
  \EulerRing
  \lra
  \H^{*}(\B\Tor(X_{\x});\QQ)\label{eq:Euler Ring for homotopy spheres}
\end{equation}
is split injective as a map of \( \Gamma' \)-representations, for some finite-index subgroup \( \Gamma'\le\Gamma \).
\cref{sec:gener-dehn-twists} could be used if one had control of the maps induced on the homology of a product of homotopy \( n \)-spheres \( \Sigma^n_1\times\Sigma^n_2 \) by diffeomorphisms coming from
\[
  \mathrm{C}^{\infty}(\Sigma^n_{1},\Diff_{\id}(\Sigma^n_{2}))\rtimes\Diff_{\id}(\Sigma^{n}_{1}).
\]
We understand the case of standard  spheres, by \cref{Dehn has monodromy P}, but a complete understanding for exotic spheres is out of reach; see \cite{OliverSommerPhD}.
However, the ideas of \cite{basuRealizingCongruenceSubgroups2016} leading up to Theorem 3.8 in loc.\ cit.\ could potentially be adapted to gain enough information, under the assumption that \( \x\ge3 \).

Since we know that \eqref{eq:Euler Ring for homotopy spheres} is injective, there is an alternative approach which one can attempt to prove that it splits equivariantly.
By using \cite{basuRealizingCongruenceSubgroups2016} and \cite[Corollary 16.4]{BassMilnorSerre1967}, one would be able to conclude that \eqref{eq:Euler Ring for homotopy spheres} splits equivariantly for \( \x\ge3 \) if the target of \eqref{eq:Euler Ring for homotopy spheres} had degreewise finite dimension.
In turn, the latter could be proven as in \cite[Corollary 5.5]{Kupers2019finiteness} if one knew that the classifying space of the Torelli group \( \pi_0\Tor(X_{\x}) \) is homologically of finite type.
Note, however, that \cite{BassMilnorSerre1967} cannot be used if \( \x \) is 2.

\section{Decomposable bundles}\label{sec:decomp-bund}

Consider the set of smooth oriented \( \SnSn \)-bundles \( \{\Theta_{\gamma}\mid\gamma\in\Gamma\} \) where
\[
  \Theta_{\gamma}\cl
  E_{\gamma}
  \lra
  B_{\gamma}
\]
is the universal smooth oriented \( \SnSn \)-bundle over the cover \( B_{\gamma} \) corresponding to the subgroup of mapping classes that induce an element in \( \langle{\gamma}\rangle \) on \( \H_n(\SnSn;\ZZ) \).
That is, \( \Theta_{\gamma} \) is classified by
\[
  \B\Diff^{\langle{\gamma}\rangle}(\SnSn)
  \lra
  \B\Diff^+(\SnSn).
\]
(See \cref{sec:class-fibr-1} for notation.)
In this section, we prove that this family of bundles detects any given characteristic class of oriented \( \SnSn \)-fibrations by using their relation to the decomposable homology classes of \cref{decomposables generation} by Goldman--Millson \cite{GoldmanMillson1986decomposable}.
If \( \gamma \) is parabolic, we can tell which characteristic classes are detected by \( \Theta_{\gamma} \).
But we pursue this in \cref{sec:E2k}, where we also explain the relation to \autocite[Example 5.3]{Morita1987t2bundles}.

\begin{theorem}\label{pairing decomposable fibrations}
  The set of smooth oriented \( \SnSn \)-bundles \( \{\Theta_{\gamma}\mid \gamma\in \Gamma\} \) satisfies the following properties.
\begin{enumerate}
    \item\label{item:decomposable detection} For each non-zero rational characteristic class \( c \) in \( \H^{*}(\B\aut^+(\SnSn);\QQ) \) there exists an element \( \gamma \) in \( \Gamma \) such that \( c(\Theta_{\gamma}) \) is non-zero.
   \item\label{item:ev for parabolic}
         If \( \gamma \) is parabolic, then there exists a rational characteristic class \( c \) such that \( c(\Theta_{\gamma}) \) is non-zero.
\end{enumerate}
\end{theorem}

\begin{proof}
  The classifying map for the underlying fibration of \( \Theta_{\gamma} \) factors as the composite in the commutative square
  \[
    \begin{tikzcd}
      \B\Diff^{\langle{\gamma}\rangle}(\SnSn)
      \ar[r]
      \ar[d]
      &
      \B\Diff^+(\SnSn)
      \ar[d]
      \\
      \B\aut^{\langle{\gamma}\rangle}(\SnSn)
      \ar[r,"\theta_{\gamma}"']
      &\B\aut^+(\SnSn).
    \end{tikzcd}
  \]
  Therefore, by \cref{injection of characteristic classes; monodromy in G}, it is enough to prove the analogous statement for the fibrations classified by \( \theta_{\gamma} \).

  By \cref{SnSn characteristic classes}, the Serre spectral sequence identifies the maps \( \H^{*}(\theta_{\gamma};\QQ) \) and \( \H_{*}(\theta_{\gamma};\QQ) \) with the restriction \( \operatorname{res}^{\Gamma}_{\langle{\gamma}\rangle} \) and the corestriction \( \operatorname{cor}^{\Gamma}_{\langle{\gamma}\rangle} \)
  \begin{gather*}
    \operatorname{cor}^{\Gamma}_{\langle{\gamma}\rangle}\cl
    \H_1(\langle{\gamma}\rangle; V)
    \lra
    \H_1(\Gamma;V)
    \\
    \operatorname{res}^{\Gamma}_{\langle{\gamma}\rangle}\cl
    \H^1(\Gamma;W)
    \lra
    \H^1(\langle{\gamma}\rangle;W)
  \end{gather*}
  respectively, where \( V=\QQ[\epx,\epy] \) and \( W=\QQ[\ex,\ey] \).
  Together with \cref{parabolic ses}, this shows Part~\ref{item:ev for parabolic}.

  To prove Part~\ref{item:decomposable detection}, let \( c\in\H^1(\Gamma;W) \) be a non-zero class. The standard pairing
  \[
    \H^1(\Gamma;W )\otimes \H_1(\Gamma;V)\lra \QQ[0]
  \]
  is perfect, so there exists a homology class for which its pairing with \( c \) is non-zero.
  As this homology class can be written as a sum of decomposable classes as in \cref{decomposables generation}, there exists one such class \( \operatorname{cor}^{\Gamma}_{\langle{\gamma}\rangle}(s^m_{\gamma}\cap\gamma) \) which pairs non-trivially with \( c \).
  By naturality of the pairing we conclude that \( \H^{*}(\theta_{\gamma};\QQ)(c) \) is non-zero.
\end{proof}

We note that \cref{pairing decomposable fibrations} implies \cref{intro:injection of characteristic classes}.
This provides a proof of \cref{intro:injection of characteristic classes} that is different from the one given in \cref{sec:injection of characteristic classes}.
The main difference is that here we use \cref{decomposables generation} by Goldman--Millson~\cite{GoldmanMillson1986decomposable}, instead of the result \cref{vcd of Gamma} about the virtual cohomological dimension of \( \Gamma \).

\begin{corollary}
    The map induced by
          \begin{equation*}
            \B\Diff^{+}(\SnSn)
            \lra
            \B\aut^{+}(\SnSn)
          \end{equation*}
          on rational cohomology is injective.
\end{corollary}

\section{Bundles detecting Eisenstein characteristic classes}\label{sec:E2k}

In this section, we make Part~\ref{intro:item:ev for parabolic} in \cref{pairing decomposable fibrations} and its relation to \cite[Example 5.3]{Morita1987t2bundles} more explicit.
In \cref{pairing decomposable fibrations} we use that the bundles \( \Theta_{\gamma} \) induce the restriction
\[
  \operatorname{res}_{\langle{\gamma}\rangle}\cl \H^1(\Gamma;\QQ[\ex,\ey])
  \lra
  \H^1(\langle{\gamma}\rangle;\QQ[\ex,\ey]),
\]
meaning that the map induced on rational cohomology by the map that classifies the underlying oriented fibration factors as \( \operatorname{res}_{\langle{\gamma}\rangle} \) followed by a monomorphism.
The second statement of \cref{pairing decomposable fibrations} is then obtained by using \cref{parabolic ses}, which says that the direct sum of restriction maps
\begin{equation*}
  \label{eq:restriction to parabolics}
  \H^1(\Gamma;\QQ[\ex,\ey]_{2m}) \lra[\oplus\operatorname{res}_{\Gamma_{x}}] \bigoplus_{\Gamma\cdot x\in\Gamma\backslash \mathrm{P}^1(\QQ)}\H^1(\Gamma_{x};\QQ[\ex,\ey]_{2m})
\end{equation*}
is surjective.
If \( n \) is 1, 3 or 7, then \( \Gamma \) is \( \SL_2(\ZZ) \), and there is only one summand.
Otherwise, there are two: those corresponding to the subgroups \( P \) and \( Q \) generated by
\[
\begin{psmallmatrix}
  1&2\\0&1
\end{psmallmatrix}
\quad\text{and}\quad
\begin{psmallmatrix}
  2&-1\\1&0
\end{psmallmatrix}
,
\]
respectively.
Note that in this case, \( \Gamma \) is also known as the theta subgroup.

Morita constructs a smooth oriented \( \S^{1}\times\S^1 \)-bundle over
\[
  M_Q^{2}(1,1)
  \lra
  \B Q\times \CP^2
\]
which induces \( \mathrm{res}_{Q} \), and defines an explicit cocycle \( \E_4 \) such that \( \mathrm{res}_{Q}(\E_4) \) is non-zero.

In this section, we construct a smooth oriented \( \SnSn \)-bundle
\[
  M_Q(1,1)
  \lra
  \B Q\times \B\SO(n+1)
\]
which induces \( \operatorname{res}_{Q} \), and we show that the smooth oriented \( \SnSn \)-bundle
\begin{equation*}\label{eq:restriction of Dehn twists}
  M_{P}
  \lra
  \B\Dehn
\end{equation*}
classified by \( \dehn \) induces \( \operatorname{res}_P \).
In particular, these bundles admit maps to \( \Theta_Q \) and \( \Theta_P \), respectively.

For each natural number \( k \), we construct a rational characteristic class of oriented \( \SnSn \)-fibrations
\[
  \E_{2k}\in\H^1(\SL_2(\ZZ);\QQ[\ex,\ey])
\]
of degree \( 2kn+2k+1 \).
When \( n \) is 1, this recovers the class in Morita's example.
This section ends with the definition of \( \E_{2k} \), and the proof of the following lemma.
\begin{lemma}\label{E2k nonzero in BDehn}
  The class \( \E_{2k} \) is non-zero when restricted along a parabolic subgroup.
\end{lemma}
\begin{remark}
  In view of \cref{parabolic ses}, \cref{E2k nonzero in BDehn} implies that, after extension to complex coefficients, and modulo parabolic cohomology classes and scaling, \( \E_{2k} \) agrees with the class corresponding to the Eisenstein series of weight \( 2k+2 \) under the Eichler--Shimura isomorphism.
  An explicit cocycle for the latter is given in \cite[\sec 1.4]{haberlandPeriodenModulformenVariabler1983}.
\end{remark}

The results of this section thus show the following.

\begin{proposition}
  For each natural number \( k \), the classes \( \mathscr{E}_{2k}(M_Q(1,1)) \) and \( \E_{2k}(M_{P}) \) are non-zero.
\end{proposition}

When \( n \) is \( 3 \), we show that we can restrict \( M_{P} \) to a bundle
\[
  M_{P}^{\ell}(1,0)
  \lra
  \B P\times \HP^{\ell}
\]
that detects the classes \( \E_{2k} \) for \( 2k\le \ell  \).

\paragraph{The bundle of Morita}

Using the Lie group structure on \( \S^1 \) we can define an action of a matrix in \( \SL_2(\ZZ) \) on \( \S^1\times\S^1 \) by the orientation-preserving diffeomorphism
\[
\begin{psmallmatrix}
  a&b\\c&d
\end{psmallmatrix}
\cdot(x,y)
=(x^ay^b,x^cy^d).
\]
By applying homology in degree \( 1 \) we recover the matrix.
In fact, since \( \S^1 \) is abelian, this defines an action of the entire special linear group.
Morita uses this action of \( \begin{psmallmatrix}
  2&-1\\1&0
\end{psmallmatrix}
\)
to define an automorphism of the smooth oriented \( \S^1\times\S^1 \)-bundle \( E(1,1) \) associated to the Whitney sum of the canonical line bundle over \( \CP^2 \) with itself.
Morita shows that the mapping torus
\begin{equation}
  \label{eq:Morita's bundle}
  M_Q(1,1)\lra \S^1\times\CP^2
\end{equation}
of the bundle \( E(1,1) \) detects the class \( \E_4 \).
In terms of clutching functions, it may be constructed as follows.
By letting each \( \S^1 \) factor act on itself by left multiplication, we get a map
\begin{equation}
  \label{eq:chi for S1}
  \chi \cl \S^1\times\S^1 \lra \Diff^+(\S^1\times\S^{1})
\end{equation}
which is equivariant for the conjugation action of \( \SL_2(\ZZ) \) on the diffeomorphism group.
In particular, the diagonal in \( \S^1\times\S^1 \) is fixed pointwise under the conjugation action of \( Q \) on the target.
We obtain a map
\begin{equation}
  \label{eq:classifying map of Morita's bundle}
  \B Q\times \B\S^{1}
  \lra
  \B\Diff^+(\S^1\times\S^1)
\end{equation}
which classifies \eqref{eq:Morita's bundle} when restricted along the map \( \CP^2\ra\B\S^1 \) classifying the canonical line bundle.

\paragraph{The generalised bundle $M_Q(1,1)$}
The lift of \( Q \) to self-diffeomorphisms of \( \S^1\times\S^1 \) given above is defined using the Lie group structure of \( \S^1 \), and the invariance of the diagonal under conjugation by this lift holds because \( \S^{1} \) is abelian.
This is special to the case in which \( n \) is 1.
For our generalisation of the bundle \( M_Q(1,1) \) we pick a different diffeomorphism lift of \(
\begin{psmallmatrix}
  2&-1\\1&0
\end{psmallmatrix}
\).
Namely, the self-diffeomorphism of \( \SnSn \) with value \( (\theta(x,y),x) \) at \( (x,y) \), where
\[
  \theta(x,y)=y-2\langle{x,y}\rangle x
\]
and \( \langle{-,-}\rangle \) is the inner product of \( \RR^{n+1} \).
The map \( \theta(-,y) \) is the reflection along the plane orthogonal to \( y \), and it has degree 2.
See \cite[p.14]{Steenrod1962cohomologyOperations}.
For all \( n \) the image of the map
\begin{equation*}
  \label{eq:Whitney sum}
  \SO(n+1)
  \lra[\Delta]
  \SO(n+1)^{\times 2}
  \lra[\chi]
  \Diff^+(\SnSn)
\end{equation*}
is fixed under conjugation by the given diffeomorphism.
Indeed, for a diffeomorphism \( \phi\times \phi \) in the image of \( \chi\circ\Delta \) we have
\[
  \phi(\theta(x,y))=\theta(\phi(x),\phi(y))
\]
because \( \phi \) is an orthogonal transformation.
We obtain a map
\begin{equation}
  \label{eq:classifying map for general MQ(1,1)}
  \B Q\times \B\SO(n+1)
  \lra
  \B\Diff^+(\SnSn).
\end{equation}
Denote by \( M_Q(1,1) \) the bundle classified by it.
\begin{lemma}\label{MQ induces resQ}
  The bundle \( M_Q(1,1) \) induces \( \operatorname{res}_Q \).
\end{lemma}
\begin{proof}
  On rational cohomology, the map classifying the underlying fibration of \( M_Q(1,1) \) factors through
  \[
    \begin{tikzcd}[column sep=large]
      \H^1(\Gamma;\QQ[\ex,\ey])
      \ar[r,"\operatorname{res}_{Q}"]
      &\H^1(Q;\QQ[\ex,\ey])
      \ar[r,"\H^{1}(Q;\Delta^{*})"]
      &\H^1(Q;\QQ[e]).
    \end{tikzcd}
  \]
  The map \( \H^1(Q;\Delta^{*}) \) is an isomorphism because \( \Delta^{*} \)  identifies \( \QQ[e] \) with the coinvariants \( \QQ[\ex,\ey]_{Q} \).
\end{proof}

\paragraph{The bundle $M_P$}
The map \( \dehn \) given by generalised Dehn twists induces the map
\[
  \B\Dehn
  \lra
  \B\Diff^+(\SnSn)
\]
which classifies the bundle that we denote by \( M_P \).
See \cref{bundles and dehn twists} for a bundle-theoretic description of \( M_P \).

\begin{lemma}\label{MP induces resP}
  The bundle \( M_P \) induces \( \operatorname{res}_P \).
\end{lemma}
\begin{proof}
  The map
  \[
    \B\Dehn
    \lra
    \B\hDehn
  \]
  admits a rational section by \cref{model for Dehn spaces}.
  By \cref{Dehn has monodromy P} we have a diagram
  \begin{equation}\label{eq:fibration over space of smooth Dehn twists}
    \begin{tikzcd}
      \B\hDehn
      \ar[d]
      \ar[r]
      &\B\aut^+(\SnSn)
      \ar[d]
      \\
      \B P
      \ar[r]
      &\B\Gamma.
    \end{tikzcd}
  \end{equation}
  We extract a rational model for the vertical map on the left from \cref{monodromy P action on Dehn space}, and we deduce the following.
  The rational cohomology of its homotopy fibre is the polynomial algebra in the Euler class \( \ex \) with the trivial action of \( P \).
  The map between the vertical homotopy fibres of \eqref{eq:fibration over space of smooth Dehn twists} induces the projection \( \QQ[\ex,\ey]\ra \QQ[\ex] \), which identifies \( \QQ[\ex] \) as the \( P \)-coinvariants.
  As in the proof of \cref{MQ induces resQ}, this concludes the proof.
\end{proof}

\paragraph{Quaternions}

We show that \( M_P \) restricts to a bundle \( M_P(1,0) \) over \( \B P\times \B\S^3 \) that still detects the classes \( \mathscr{E}_{2k} \).
When restricted to \( \B P\times\HP^{\ell} \), it detects the classes \( \mathscr{E}_{2k} \) for \( 2k\le \ell  \).

Observe that \( \mathrm{C}^{\infty}(\S^3,\SO(4)) \) fibres over \( P \), by \cref{Dehn has monodromy P}.
We consider the section that sends the generator \( \mathrm{t}=
\begin{psmallmatrix}
  1&1\\0&1
\end{psmallmatrix}
\) to the map \( \S^3\ra\SO(4) \) given by right multiplication for the Lie group structure on \( \S^3 \).
We let \( \mathrm{t} \) act fibrewise by conjugation.

The action of \( \S^3 \) on itself by left multiplication gives a map of topological groups \( \ell\cl\S^3\ra\SO(4) \).
Consider the map
  \begin{equation}
  \S^3\lra\mathrm{C}^{\infty}(\S^3,\SO(4))\label{eq:S3 bundle}
\end{equation}
which sends \( \omega \) to the constant map at \( \ell(\omega) \).
Its image is fixed pointwise under the conjugation action of \( \mathrm{t} \), because the multiplication of \( \S^3 \) is associative.
We obtain a map to it from the product \( P\times \S^3 \), and an induced map
\[
  \B P \times \B\S^3
  \lra
  \B\mathrm{C}^{\infty}(\S^3,\SO(4))
  \lra[\dehn]
  \B\Diff^+(\S^3\times\S^3).
\]
which classifies the bundle that we denote by \( M_P(1,0) \).

Similarly to \cref{MQ induces resQ} and \cref{MP induces resP} one proves that \( M_P(1,0) \) induces \( \operatorname{res}_P \).
Therefore, the classes \( \mathscr{E}_{2k}(M_P(1,0)) \) are non-zero.
For this, one uses that the rational cohomology of \( \B\S^3 \) is a polynomial algebra in the Euler class of the universal principal \( \S^{3} \)-bundle.

\paragraph{The classes \( \E_{2k} \)}

We define representatives of the classes \( \E_{2k} \) in \( \H^1(\Gamma;\QQ[\ex,\ey]) \) as derivations.
(Recall from \cref{sec:group-cohom-homol} the description of \( 1 \)-cocycles in terms of derivations.)
We view these classes as rational characteristic classes of oriented \( \SnSn \)-fibrations using the isomorphism
\[
  \widetilde{\H}^{*}(\B\aut^+(\SnSn);\QQ)\cong\H^1(\Gamma;\QQ[\ex,\ey])
\]
in \cref{SnSn characteristic classes}.
In the notation we do not distinguish between the class, its representative, and the corresponding cohomology class of \( \B\aut^+(\SnSn) \).
Consider the presentation
\[
  \SL_2(\ZZ)=\langle{\alpha,\beta\;|\;\alpha^4=\alpha^2\beta^{-3}=1}\rangle
\]
obtained by letting \( \alpha=
\begin{psmallmatrix}
  0 & 1\\
  -1 & 0
\end{psmallmatrix} \) and \( \beta=
\begin{psmallmatrix}
  0 & 1\\
  -1 & 1
\end{psmallmatrix} \).

\begin{lemma}\label{definition of E2k}
  For each natural number \( k \) there is a unique derivation
  \[
    \E_{2k}\cl\SL_2(\ZZ)\lra \QQ[\ex,\ey]
  \]
  such that \( \E_{2k}(\alpha)=\ex^{2k}-\ey^{2k} \) and \( \E_{2k}(\beta)=0 \).
\end{lemma}
\begin{proof}
  There is a unique extension of \( f=\E_{2k} \) to a derivation on the free group on \( \alpha \) and \( \beta \). To extend it to \( \SL_2(\ZZ) \) it is enough to check that \( f(\alpha^4)=f(1) \) and \( f(\alpha^2)=f(\beta^3) \); indeed, see~\cite[90]{Brown1994groupCohomology}. Using the derivation condition we rewrite these equations as
  \begin{gather}
    \label{eq:alpha crossed}(1+\alpha+\alpha^2+\alpha^3)f(\alpha)=0\\
    \label{eq:beta crossed}(1+\alpha)f(\alpha)=(1+\beta+\beta^2)f(\beta).
  \end{gather}
  Note that \( \alpha^2=-\id \) acts trivially on homogeneous polynomials of even degree.
  Therefore, the left-hand side of \eqref{eq:alpha crossed} takes the form
  \[
    2(1+\alpha)f(\alpha)=2(\ex^{2k}-\ey^{2k} + \ey^{2k}-\ex^{2k}) =0.
  \]
  We defined \( f(\beta)=0 \), so both equations \eqref{eq:alpha crossed} and \eqref{eq:beta crossed} are satisfied.
\end{proof}

\begin{lemma}\label{E2k nonzero in BDehn (proof)}
  The class \( \E_{2k} \) is non-zero when restricted along a parabolic subgroup.
\end{lemma}
\begin{proof}
  It is enough to check that it is non-zero when restricted a single parabolic subgroup, because up to finite-index they are all conjugate in \( \SL_2(\ZZ) \), and the coefficients are rational.
  We calculate the pullback of \( \E_{2k} \) under the homomorphism
  \[
    \H^1(\Gamma;\QQ[\ex,\ey]) \lra \Hom(P,\QQ[\ex])
  \]
  given by \( \operatorname{res}_P \) followed by the projection \( \mathrm{pr}_{x}\cl\QQ[\ex,\ey]\ra\QQ[\ex] \), where \( \QQ[\ex] \) is given the trivial action.
  The class \( \E_{2k} \) is sent to the homomorphism \( \mathrm{pr}_{x}\circ\E_{2k}\circ \operatorname{res}_{P} \), where \( \E_{2k} \) now denotes a derivation representative.
  Observe that \(
  \begin{psmallmatrix}
    1&1\\0&1
  \end{psmallmatrix}
  =\beta^{-1}\alpha \).
  Then we compute
  \begin{align*}
    \E_{2k}
    (
    \begin{psmallmatrix}
      1&1\\0&1
    \end{psmallmatrix}
    )
    &=\beta^{-1}\cdot \E_{2k}(\alpha)\\
                        &=(\beta^{-1} \ex)^{2k}-(\beta^{-1} \ey)^{2k}\\
                        &=\ey^{2k}-(\ey-\ex)^{2k}
  \end{align*}
  using the derivation property.
  This element is sent to \( -\ex^{2k} \) by \( \mathrm{pr}_x \).
\end{proof}

\printbibliography
\contact{Department of Mathematics, Stockholm University, SE-106 91 Stockholm, Sweden}{jan.mcgarry.furriol@math.su.se}
\end{document}